\documentclass[journal]{IEEEtran}  

\IEEEoverridecommandlockouts                              

\usepackage[utf8]{inputenc} 
\usepackage[T1]{fontenc}    
\usepackage{hyperref}       
\usepackage{balance}
\usepackage{url}            
\usepackage{booktabs}       
\usepackage{amsfonts}       
\usepackage{nicefrac}       
\usepackage{microtype}      
\usepackage{enumitem,wrapfig}
\usepackage{algorithm}		
\usepackage{algpseudocode}	
\usepackage[small]{caption}
\setlist{leftmargin=4.2mm}

\usepackage[backend=bibtex,style=numeric-comp,sorting=none]{biblatex}
\usepackage{amsmath,amssymb,amsthm}
\usepackage{graphicx}

\usepackage{xcolor}

\newtheorem{theorem}{Theorem}
\newtheorem{lemma}{Lemma}
\newtheorem{proposition}{Proposition}

\newtheorem{definition}{Definition}
\newtheorem{assumption}{Assumption}

\newcommand{\subparagraph}{} 

\newcommand{\argmin}{\operatornamewithlimits{arg\,min}}

\newcommand{\learningratehum}{\eta^\text{h}}
\newcommand{\learningrateaut}{\eta^\text{a}}
\newcommand{\demandhumRV}{\lambda^\text{h}}
\newcommand{\demandautRV}{\lambda^\text{a}}
\newcommand{\demandhum}{\bar{\lambda}^\text{h}}
\newcommand{\demandaut}{\bar{\lambda}^\text{a}}

\newcommand{\pathidx}{p}
\newcommand{\pathset}{\mathcal{P}}
\newcommand{\cellidx}{i}
\newcommand{\cellset}{\mathcal{I}}

\newcommand{\timeidx}{k}

\newcommand{\ffl}{a}

\newcommand{\numnbcells}{m^\text{n}_\pathidx}
\newcommand{\numbcells}{m^\text{b}_\pathidx}

\newcommand{\nbcellsset}{{[}\numnbcells{]}}
\newcommand{\numlanesnb}{\numlanes^\text{n}}
\newcommand{\numlanesb}{\numlanes^\text{b}}
\newcommand{\laneratio}{r}

\newcommand{\ffvelocity}{\bar{v}}
\newcommand{\numlanes}{b}
\newcommand{\spacehuman}{h^\text{h}}
\newcommand{\spaceaut}{h^\text{a}}

\newcommand{\send}{S}
\newcommand{\recv}{R}
\newcommand{\priority}{q}

\newcommand{\fracouthum}{\beta^\text{h}}
\newcommand{\fracoutaut}{\beta^\text{a}}
\newcommand{\upstreamcells}{\mathcal{U}}

\newcommand{\dens}{n}
\newcommand{\denshum}{n^\text{h}}
\newcommand{\densaut}{n^\text{a}}
\newcommand{\denscong}{n^\text{c}}
\newcommand{\flow}{f}
\newcommand{\flowhum}{f^\text{h}}
\newcommand{\flowaut}{f^\text{a}}
\newcommand{\flowin}{y}
\newcommand{\flowinhum}{y^\text{h}}
\newcommand{\flowinaut}{y^\text{a}}
\newcommand{\routinghum}{\mu^\text{h}}
\newcommand{\routingaut}{\mu^\text{a}}

\newcommand{\jamden}{\bar{n}}
\newcommand{\jamdennb}{\bar{n}^\text{n}}
\newcommand{\jamdenb}{\bar{n}^\text{b}}
\newcommand{\critdens}{\tilde{n}}
\newcommand{\critdensnb}{\tilde{n}^\text{n}}
\newcommand{\critdensb}{\tilde{n}^\text{b}}
\newcommand{\capacity}{\bar{F}}
\newcommand{\capacitynb}{\bar{F}^\text{n}}	
\newcommand{\capacityb}{\bar{F}^\text{b}}	
\newcommand{\shockspd}{w}
\newcommand{\shockspdnb}{w^\text{n}}
\newcommand{\shockspdb}{w^\text{b}}
\newcommand{\autlev}{\alpha}

\newcommand{\latency}{\ell}
\newcommand{\latencycong}{\ell^{\text{c}}}

\title{\LARGE \bf
	Learning How to Dynamically Route \\ Autonomous Vehicles on Shared Roads
}

\author{Daniel A.~Lazar$^*$, Erdem B\i y\i k$^*$, Dorsa~Sadigh, Ramtin~Pedarsani\vspace{-30px}
	\thanks{$^{*}$Authors contributed equally.}
	\thanks{Daniel Lazar is with the Department of Electrical and Computer Engineering, 
		UC Santa Barbara
		{\tt\small dlazar@ece.ucsb.edu}}
	\thanks{Erdem B\i y\i k is with the Department of Electrical Engineering,
		Stanford University
		{\tt\small ebiyik@stanford.edu}}
	\thanks{Dorsa~Sadigh is with the Departments of Computer Science and Electrical Engineering,
		Stanford University
		{\tt\small dorsa@cs.stanford.edu}}
	\thanks{Ramtin~Pedarsani is with the Department of Electrical and Computer Engineering, 
		UC Santa Barbara
		{\tt\small ramtin@ece.ucsb.edu}}
	}

\begin{document}
\abovedisplayskip=5pt
\abovedisplayshortskip=5pt
\belowdisplayskip=5pt
\belowdisplayshortskip=5pt
	
\maketitle
\thispagestyle{empty}
\pagestyle{empty}

\begin{abstract}
	
  Road congestion induces significant costs across the world, and road network disturbances, such as traffic accidents, can cause highly congested traffic patterns. If a planner had control over the routing of all vehicles in the network, they could easily reverse this effect. In a more realistic scenario, we consider a planner that controls autonomous cars, which are a fraction of all present cars. We study a dynamic routing game, in which the route choices of autonomous cars can be controlled and the human drivers react selfishly and dynamically. As the problem is prohibitively large, we use deep reinforcement learning to learn a policy for controlling the autonomous vehicles. This policy indirectly influences human drivers to route themselves in such a way that minimizes congestion on the network. To gauge the effectiveness of our learned policies, we establish theoretical results characterizing equilibria and empirically compare the learned policy results with best possible equilibria. We prove properties of equilibria on parallel roads and provide a polynomial-time optimization for computing the most efficient equilibrium. Moreover, we show that in the absence of these policies, high demand and network perturbations would result in large congestion, whereas using the policy greatly decreases the travel times by minimizing the congestion. To the best of our knowledge, this is the first work that employs deep reinforcement learning to reduce congestion by indirectly influencing humans' routing decisions in mixed-autonomy traffic.
	
\end{abstract}

\renewcommand\IEEEkeywordsname{Keywords}

\vspace{-15px}
\section{Introduction}
\label{sec:introduction}
\IEEEPARstart{C}{ongestion} can result in substantial economic and social costs~\cite{schrank2015} which have only been growing in recent years, especially with the advent of ride-hailing services~\cite{henao2017impacts,beojone2021inefficiency}. Congestion is formed by a number of mechanisms, such as when many vehicles try to enter a road at the same time. A higher-level cause is from how people choose their routes -- when people selfishly choose the quickest routes available to them, this often results in greater congestion and longer travel time than if people had their routes chosen for them optimally in terms of the overall experienced delay~\cite{roughgarden2002bad}. There are some existing methods for fighting congestion, such as congestion pricing~\cite{hu2019newYork}, variable speed limits~\cite{lu2011novel} and highway ramp metering~\cite{gomes2008behavior}. However, they can be difficult to administer, and can require significant changes to infrastructure.

The introduction of autonomous vehicles to public roads provides an opportunity for better congestion management \cite{di2020survey}.
Our key idea is that by controlling the routing of autonomous vehicles, we can change the delay associated with traversing each road, thereby indirectly influencing peoples' routing choices. By influencing people to use more ``socially advantageous'' routes, we can eliminate long queues and significantly reduce traffic jams on roads.

The model for \emph{mixed-autonomy} traffic, meaning traffic with both human-driven and autonomous vehicles, is complex, involving very large and continuous state space and continuous action space. Having human drivers dynamically respond to the choices of the autonomous vehicles further complicates the matter, making a dynamic programming-based approach and other classical methods infeasible. Because of this, we use model-free deep reinforcement learning (RL) to learn a policy without requiring access to the dynamics of the transportation network. Specifically, we show it is possible to learn a policy via proximal policy optimization (PPO) \cite{schulman2017proximal} that mitigates traffic congestion by managing routing of autonomous cars given the network state. 

To understand the performance of the learned policy, we investigate the \emph{equilibrium} behavior of the network. Previous works \cite{krichene2017stackelberg, lazar2018altruism} have shown that there is a wide spectrum of equilibria in traffic networks, meaning situations in which everyone is taking the quickest route immediately available to them, and these equilibria can have greatly varying average user delay. We establish efficient ways to compute equilibria in the network and compare the best equilibrium (in terms of latency) with the RL policy, which works regardless of whether equilibrium conditions hold or not. We show that the learned policy reaches the `desirable' equilibria that have low travel times when starting with varying traffic patterns, and can recover network functionality after a disturbance such as a traffic accident. To summarize, our contributions are as follows:

\begin{figure*}[t]
	\centering
	\includegraphics[width=\linewidth]{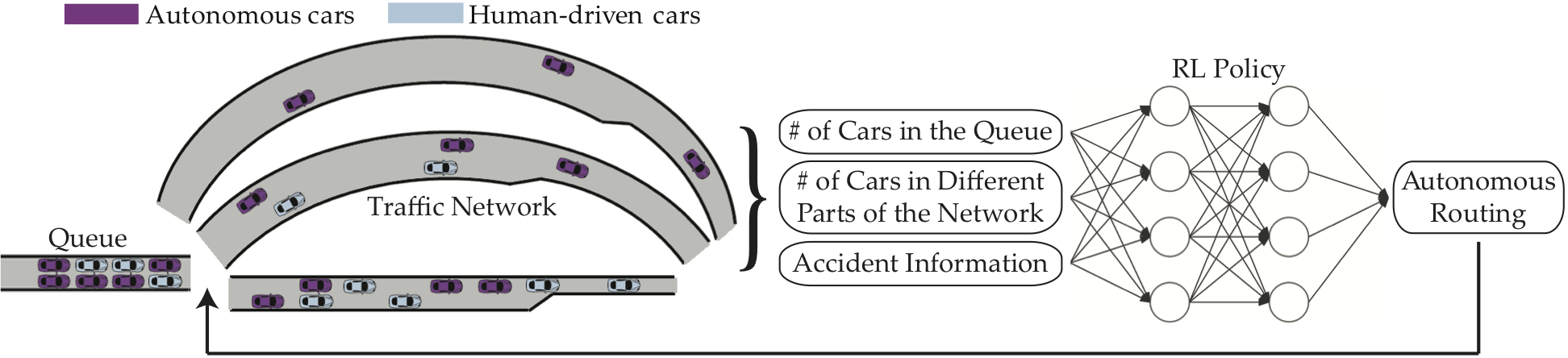}
	\vspace{-15px}
	\caption{The schematic diagram of our framework. Our deep RL agent processes the state of the traffic and outputs a control policy for autonomous cars' routing.}
	\label{fig:schematic}
	\vspace{-20px}
\end{figure*}

\begin{itemize}[nosep]
	\item \emph{Theoretical analysis:} We characterize equilibria in the network and derive a polynomial-time computation for finding optimal equilibria of parallel networks. 
	\item \emph{Finding a control policy via deep RL:} We employ deep RL methods to learn a routing policy for autonomous cars that effectively saves the traffic network from unboundedly large delays. We show via simulation that the RL policy is able to bring our network to the best possible equilibrium when starting from a congested state or after a network disturbance on parallel networks. We further show that an MPC-based approach and a greedy optimization method fail to do so, and thus is outperformed by the RL-based method in general networks.
\end{itemize}

We visualize our framework in the schematic diagram Fig.~\ref{fig:schematic}.

\textbf{Literature review.} Many works seek to understand how much traffic network latency could be improved if vehicle routing was controlled by a central planner, including works on \emph{congestion games} \cite{dafermos1972multiclass_user, hearn1984convex, roughgarden2002bad, lazar2020routing, mehr2018can}. Some study how indirectly influencing peoples' routing choices by providing them network state information affects network performance \cite{lazarus2018decision, wu2018value}. \emph{Stackelberg Routing}, in which only some of the vehicles are controlled, is another way to influence routing \cite{roughgarden2004stackelberg, swamy2012effectiveness}; some works incorporate the \emph{dynamics} of human routing choices \cite{krichene2018social}. While providing useful techniques for analysis, the congestion game framework does not reflect a fundamental empirical understanding about vehicle flow on roads, namely that roads with low vehicle density have a roughly constant latency, and roads with high density see latency \emph{increase} as flow \emph{decreases}.

Works on CTM \cite{daganzo1994cell, muralidharan2009freeway} capture this phenomenon, including works that characterize equilibria on roads described with CTM \cite{gomes2008behavior}. Notably, some consider equilibria of parallel-path Stackelberg Games, including with mixed autonomy \cite{krichene2017stackelberg,lazar2018altruism}. However, their analyses are limited to steady-state and do not capture the dynamics. \cite{aswani2011game} considers a Fundamental Diagram of Traffic-based model for slowly varying traffic. They formulate this as a Stackelberg Game and design routing information for users to minimize overall latency and bound the resulting inefficiency in a simple network. However, they only consider a single-vehicle type, not a mixed-autonomy setting.

Some works look at the low-level control of autonomous cars, specifically controlling acceleration to smooth flow and ease congestion at bottlenecks \cite{cui2017stabilizing, wu2017emergent, wu2018stabilizing}; \cite{vinitsky2018benchmarks} provides a benchmark for gauging the performance of these techniques. Other works learn ramp metering policies \cite{belletti2018expert}, localize congestion \cite{sivaranjani2015localization}, and model lane-change behavior with a neural network \cite{wright2019neural}.

In addition to these learning methods, there has also been an effort to use RL for route selection \cite{mao2018reinforcement} and driver choice modeling in traffic assignment problem \cite{bazzan2016multiagent,zhou2020reinforcement,ramos2018analysing,stefanello2016using}. Again using RL, \cite{grunitzki2014individual} shows reward shaping mechanisms could be utilized to reach better equilibria. Recently, \cite{shou2020reward,shou2020multi} develop a hierarchical approach to optimize fares, tolling and signal control in the high-level whereas a multi-agent RL method models the drivers in the lower level. Although these works show the effectiveness and potential of RL methods in transportation, to the best of our knowledge, these methods have not been used in a routing game with mixed-autonomy traffic where a central planner aims to reduce congestion by indirectly influencing humans' routing via the routing of autonomous vehicles.

Without any reinforcement learning component, some works provide macroscopic models of roads shared between human-driven and autonomous cars. \cite{li2018modeling} models highway bottlenecks in the presence of platoons of autonomous vehicles mixed in with human-driven vehicles. The authors relate their model to a CTM type-model similar to the model presented below, though it is specific to a single highway. \cite{mahmassani201650th} describes a microscopic model to determine the effect of autonomy on throughput, yielding fundamental diagrams. The fundamental relationship between autonomy level and critical density in our model mirrors that of \cite{levin2016multiclass}, which develops a CTM model for mixed autonomy traffic.

Some works solve the dynamic traffic assignment problem for networks with a CTM-based flow model, including some which decompose the optimization to enable optimizing flow on large networks \cite{mehrabipour2019decomposition}. In contrast, our works studies the setting in which some flow demand is controlled to optimize the system performance, and some flow demand updates according to a selfish update rule. This precludes the use of such decomposition techniques, since the optimization can no longer be formulated as a linear program. Because of this, we use RL to solve for a routing policy in our setting.

\vspace{-5px}
\section{Vehicle flow dynamics: modeling roads}
\label{sec:flow_dynamics}
In this section we describe dynamics governing how vehicle flow travels on a road. We extend the CTM, a widely used model that discretizes roads into \emph{cells}, each with uniform density \cite{daganzo1994cell, muralidharan2009freeway}, for mixed-autonomy traffic. In CTM, each road segment has a maximum flow that can traverse it. The key idea of our extension is that since autonomous vehicles can keep a shorter headway (distance to the car in front of it), the greater the fraction of autonomous vehicles on a road, the greater the maximum flow that the road can serve \cite{lazar2018altruism}. Accordingly, our extension of CTM lies in the dependence of cell parameters on the \emph{autonomy level}, or the fraction of autonomous vehicles, in each cell. 

We use our capacity model in conjunction with Daganzo's CTM formulation in \cite{daganzo1994cell, daganzo1995cell}, the combination of which we describe in the following. We consider a network of roads with a single origin and destination for all vehicles in the network. The origin and destination are connected by the set of simple paths $\pathset$. Each path is composed of a number of cells, and we denote the set of cells composing path $\pathidx$ by $\cellset_\pathidx$. We generally use $\cellidx$ and $\pathidx$ as indices for cells and paths, respectively.

\begin{figure*}[t]
	\centering
	\includegraphics[width=0.8\linewidth]{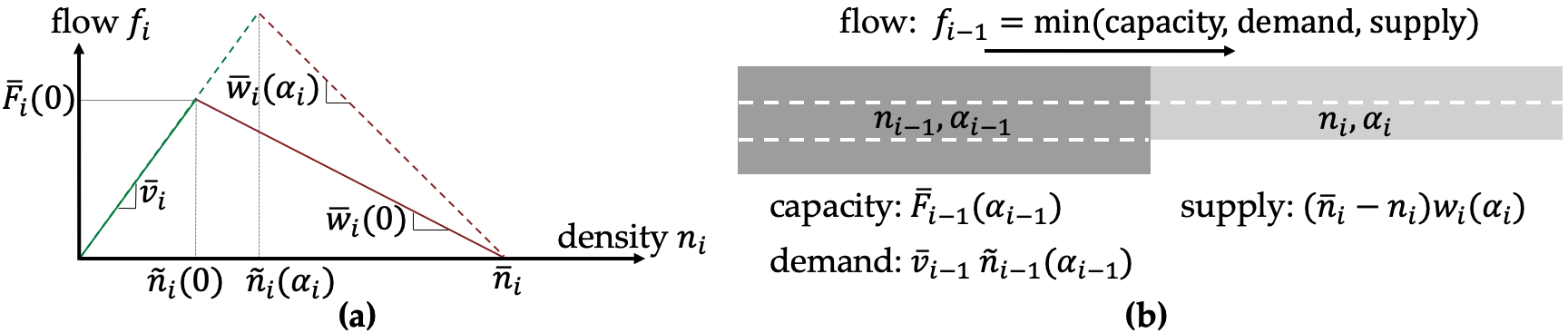}
	\caption{\textbf{(a)} Fundamental diagram of traffic governing vehicle flow in each cell of the Cell Transmission Model. The solid line corresponds to a cell with only human-driven vehicles; the dashed line represents a cell with both vehicle types at autonomy level $\autlev_\cellidx$. Green and red respectively represent a cell in free-flow and congestion. \textbf{(b)} The flow from one cell to another is a function of the density $\dens$ and autonomy level $\autlev$ in each cell. In both figures, we suppress the notation for path $\pathidx$.}
	\label{fig:FDTs}
	\vspace{-12px}
\end{figure*}

In the CTM, every cell has a \emph{critical density}, and when the density of a cell exceeds the critical density, that cell is \emph{congested}. We model the critical density as being dependent on the autonomy level. This is because autonomous vehicles maintain a different nominal headway than human-driven vehicles; in other words, autonomous vehicles may require more space in front of them due to prediction error, or less space, as they may react faster than human drivers. Accordingly, we use the model in \cite{lazar2017routing} to model the capacity of a cell.

Using this model, each cell $\cellidx$ has a free-flow velocity,  $\ffvelocity_{\cellidx}$, as well as a nominal headway for vehicles traveling at the free-flow velocity --- $\spacehuman_{\cellidx}$ cells/vehicle for human-driven vehicles and $\spaceaut_{\cellidx}$ for autonomous vehicles. The capacity of the cell then varies with the autonomy level, denoted $\autlev_{\cellidx} \in [0,1]$. We use $\numlanes_{\cellidx}$ to denote the number of lanes in a cell. We model vehicles as slowing down when the headway experienced decreases below the nominal headway required and accordingly model the critical density as follows, as in \cite{askari2017effect, lazar2017routing, lazar2018altruism, mehr2018can}:
\begin{equation}\label{eq:crit_dens}
	\critdens_{\cellidx}(\autlev_{\cellidx}) :=  \numlanes_{\cellidx}/(\autlev_{\cellidx}\spaceaut_\cellidx + (1-\autlev_{\cellidx})\spacehuman_\cellidx).
\end{equation}
Each cell also has a vehicle density, $\dens_{\cellidx} = \denshum_{\cellidx} + \densaut_{\cellidx}$, where $\denshum_{\cellidx}$ and $\densaut_{\cellidx}$ are, respectively, the number of human-driven and autonomous vehicles. Thus, $\autlev_{\cellidx} = \densaut_{\cellidx}/(\denshum_{\cellidx}+\densaut_{\cellidx})$. As the cells are very large compared to the vehicles, we consider these quantities to be continuous variables. As mentioned above, CTM has two regimes for vehicle flow: free-flow, when cell density is less than the critical density, and congestion, when cell density is greater than the critical density but less than the \emph{jam density} $\jamden_{\cellidx}$, the density at which flow  stops completely.

Three factors limit the flow from one cell to another. One is the \emph{capacity}, or maximum flow out of a cell, which is the flow of vehicles that traverse the cell at the critical density:
\begin{equation}\label{eq:capacity}
	\capacity_{\cellidx}(\autlev_{\cellidx}) := \ffvelocity_{\cellidx}  \critdens_{\cellidx}(\autlev_{\cellidx}) \; .
\end{equation}
The flow out of a cell is limited by the sending function of that cell, which is the minimum of the capacity of the cell and the \emph{demand} of vehicles in the cell: $\send_\cellidx(\autlev_\cellidx(\timeidx)) = \min( \capacity_{\cellidx}(\autlev_{\cellidx}), \ffvelocity_{\cellidx}\dens_{\cellidx}(\timeidx))$. The flow entering a cell is limited by that cell's receiving function, which is the minimum of its capacity and its \emph{supply} of vehicles: $\recv_\cellidx(\autlev_\cellidx(\timeidx)) = \min(\capacity_{\cellidx}(\autlev_{\cellidx}),  (\jamden_{\cellidx}- \dens_{\cellidx})\shockspd_{\cellidx}(\autlev_{\cellidx}))$, where $\shockspd_{\cellidx}$ is the \emph{shockwave speed}, the speed at which slowing waves of traffic propagate upstream: $\shockspd_{\cellidx}(\autlev_{\cellidx}):=\ffvelocity_{\cellidx}\critdens_{\cellidx}(\autlev_{\cellidx})/(\jamden_{\cellidx}-\critdens_{\cellidx}(\autlev_{\cellidx})).$ In the following, we use $\flow_\cellidx(\timeidx)$ to denote the flow out of cell $\cellidx$ at time $\timeidx$ and $\flowin_\cellidx(\timeidx)$ to denote the flow into cell $\cellidx$. We use the standard superscripts for human-driven and autonomous flow, with the relationships $\flowhum_\cellidx(\timeidx) + \flowaut_\cellidx(\timeidx) = \flow_\cellidx(\timeidx)$ and $\flowinhum_\cellidx(\timeidx) + \flowinaut_\cellidx(\timeidx) = \flowin_\cellidx(\timeidx)$. Accordingly,
\begin{align}\label{eq:dens_update}
	\denshum_{\cellidx}(\timeidx\!+\!1) & 
	=\denshum_{\cellidx}(\timeidx)\! +\! \flowinhum_\cellidx(\timeidx) \!- \! \flowhum_{\cellidx}(\timeidx) \; , \nonumber \\
	\densaut_{\cellidx}(\timeidx\!+\!1) & = \densaut_{\cellidx}(\timeidx)\! +\! \flowinaut_\cellidx(\timeidx) \! -\! \flowhum_{\cellidx}(\timeidx)  \; .
\end{align}

Since some cells might be a part of more than one path, we also track the paths of the human-driven and autonomous vehicles in each cell. We use $\routinghum_{\cellidx}(\pathidx, \timeidx)$ and $\routingaut_{\cellidx}(\pathidx, \timeidx)$ to denote the fraction of human-driven and autonomous vehicles, respectively, in cell $\cellidx$ at time $\timeidx$ that are taking path $\pathidx$. If cell $\cellidx$ is not on path $\pathidx$, let $\routinghum_{\cellidx}(\pathidx, \timeidx) = \routingaut_{\cellidx}(\pathidx, \timeidx) = 0$.

Extending the development in \cite{blubook_vol1_v085}, we formulate a calculation of the flow of mixed autonomous vehicles through general junctions. We define $O$ as the set of intersections, or junctions, in the network. We use $\Xi(o)$ to denote the set of turning movements through intersection $o$, with a turning movement denoted by a tuple, such as ${[}i,o,j{]}\in \Xi(o)$, where $i$ denotes the incoming cell, and $j$ denotes the outgoing cell. As before, we consider all cells to have one direction of travel. For intersection $o$ we define a set of conflict points $C(o)$, and $\Xi(c)$ denotes the set of turning movements through the intersection which pass through conflict point $c$, where $c \in C(o)$. These routes may have different priority levels, so for each ${[}i,o,j{]} \in \Xi(c)$ we define $\beta^c_{ioj} > 0 $ as the priority of turning movement ${[}i,o,j{]}$ through conflict point $c$. Each conflict point has some supply $R_c$, which we assume is independent of the level of autonomy of the vehicles passing through it. The relative priority of the turning movements will determine the relative flow of each turning movement through the conflict point. In a slight abuse of notation, we use $f_{ioj}(k)$ to denote the \emph{total} flow of vehicles through turning movement ${[}i,o,j{]}$ at time $k$; we use $f^\text{h}_{ioj}(k)$ and $f^\text{a}_{ioj}(k)$ to denote the flow of human and autonomous vehicles, respectively, through the turning movement. We use $\Gamma(o)$ and $\Gamma^{-1}(o)$ to denote the set of cells exiting and entering junction $o$, respectively.

We then calculate the flows at each time step as in Algorithm~\ref{alg:flow_calculation}.

\begin{algorithm}
	\caption{Flow Calculation}\label{alg:flow_calculation}
\begin{algorithmic}[1]
\Procedure{ Flow Calculation}{ Intersection $o$}
   \State \begin{align*}\\[-30pt]
   \forall {[}i,o,j{]}\in \Xi(o),   \; & p^\text{h}_{ioj} \gets \sum_{p \in \pathset : i,j \in \mathcal{I}_p} \mu^\text{h}_i(p,k)  \\
   	& p^\text{a}_{ioj} \gets \sum_{p \in \pathset : i,j \in \mathcal{I}_p}  \mu^\text{a}_i(p,k)  \\
   	& p_{ioj} \gets \frac{n_i^\text{h}(k) p^\text{h}_{ioj} + n_i^\text{a}(k) p^\text{a}_{ioj} }{n_i^\text{h}(k) + n_i^\text{a}(k)}
   \end{align*}
   \State \begin{align*}\\[-30pt] \forall {[}i,o,j{]} \in \Xi(o),  \; &  f_{ioj} \gets 0  \\
   	&f^\text{h}_{ioj} \gets 0  \\
   	&f^\text{a}_{ioj} \gets 0  \\
   	&\tilde{S}_{ioj} \gets S_i(\alpha_i(k)) p_{ioj}  \\
   	 \forall (o,j) \in \Gamma(o),  \; & \tilde{R}_{oj} \gets R_j(\alpha_j(k))  \\
   	 \forall c \in C(o),  \; &  \tilde{R}_c \gets R_c 
   \end{align*}
   \State For all $ {[}i,o,j{]}\in \Xi(o)$, set $\delta_{ioj}$ such that 
   \begin{align*}
   	& \forall {[}i,o,j{]}\in \Xi(o) ,  \\
	& \quad  \forall {[}i',o,j'{]}\in \Xi(o) , \; \frac{\delta_{ioj}}{\delta_{i'oj'}} = \frac{p_{ioj}}{p_{i'oj'}} \; ,  \\
	& \quad   \text{where $i$ can equal $i'$ and $j$ can equal $j'$, and}  \\
	&  \forall c \in C(o); \; \forall {[}i,o,j{]}\in \Xi(c)  \\
	& \quad   \forall {[}i',o,j'{]}\in \Xi(c) , \; \frac{\delta_{ioj}}{\delta_{i'oj'}} = \frac{\beta^c_{ioj} p_{ioj}}{\beta^c_{i'oj'} p_{i'oj'}} \; , \\
	& \quad \text{where $i$ can equal $i'$ and $j$ can equal $j'$. } 
   \end{align*} 
   \State  $A \gets \Xi(o)$ 
   \While{ $A \neq \emptyset$ }
   	\State \begin{align*}\\[-30pt]
		&  \forall (o,j) \in \Gamma(o) , \;  \delta_{oj} \gets \sum_{ {[}i,o,j{]} \in A} p_{ioj}  \\
		&  \forall c \in C(o) , \;  \delta_{c} \gets \sum_{ {[}i,o,j{]} \in \Xi(c) \cap A} p_{ioj} 
	\end{align*}
	\State \begin{align*}\\[-30pt]
		\theta = \min \{ & \min_{ {[}i,o,j{]} \in A } \frac{ \tilde{S}_{ioj} }{ \delta_{ioj} }, \min_{ (o,j) \in \Gamma(o) , \delta_{oj} > 0 } \frac{ \tilde{R}_{oj} }{ \delta_{oj} },  \\
		& \min_{ c \in C(o) : \delta_c > 0} \frac{ \tilde{R}_{c} }{ \delta_{c} } \} 
	\end{align*}
	\State \begin{align*}\\[-30pt]
		\forall {[}i,o,j{]} \in A ,  \; & f_{ioj} \gets f_{ioj} + \theta \delta_{ioj}  \\
		& f^\text{h}_{ioj} \gets f^\text{h}_{ioj} + \theta \delta_{ioj} (1- \alpha_i(k)) \\
		& f^\text{a}_{ioj} \gets f^\text{a}_{ioj} + \theta \delta_{ioj}  \alpha_i(k)  \\
		& \tilde{S}_{ioj} \gets \tilde{S}_{ioj} - \theta \delta_{ioj}  \\
		\forall (oj) \in \Gamma(o) , \; &  \tilde{R}_{oj} \gets \tilde{R}_{oj} - \theta \delta_{oj}  \\
		\forall c \in C(o) ,  \; &  \tilde{R}_{c} \gets \tilde{R}_{c} - \theta \delta_{oj} 
	\end{align*}
	\State \begin{align*}\\[-30pt]
		& A \gets A \setminus \{ {[}i,o,j{]} \in A : \tilde{S}_{ioj} = 0 \}  \\
		& A \gets A \setminus \{ {[}i,o,j'{]} \in A : \tilde{R}_{oj} = 0 \land p_{ioj'} > 0 \}  \\
		& A \gets A \setminus \{ {[}i,o,j{]} \in c : \tilde{R}_{c} = 0 \;  \forall c \in C(o) \} 
	\end{align*}
   \EndWhile
   \State \textbf{return} $f_{ioj}$, $f^\text{h}_{ioj}$, $f^\text{a}_{ioj}$, $\forall {[}i,o,j{]} \in \Xi(o)$ 
\EndProcedure
\end{algorithmic}
\end{algorithm}

An interpretation of this algorithm is as follows. The set $A$ denotes the set of turning movements with flows that can yet be increased, and each turning movement is assigned a rate at which its flow increases. As sending and receiving limits are reached, turning movements are removed from $A$ until there are no more turning movements left to increase.

In more concrete terms, first calculate the fraction of vehicles in each incoming cell which are headed to each outgoing cell. Then initialize all flows to $0$ and initialized the unused sending and receiving capacity for each cell and conflict point. We then find relative rates of flow increase, $\delta_{ioj}$, for the turning movements. In the loop, we calculate the similar rates of flow increases for the receiving cells and conflict points based on the rates previously found. Then, the flows are increased by the established rates until either a sending limit, a cell receiving limit, or conflict point capacity is reached. Any turning movement that has reached its sending limit is removed from the set of turning movements with further flow increases, $A$. Similarly, any turning movement that exits from a cell which has reached its receiving limit is removed from $A$, and the same with turning movements through conflict points which have reached their capacity. The loop repeats until $A$ is empty.

Having calculated the flow through the intersection, the states of each cell is updated as follows. We compute the incoming flows for the outgoing cells as follows:
\begin{align}\label{eq:incoming_flow}
	\forall (o,j) \in \Gamma(o) ,  \; & y_j^\text{h}(k) = \sum_{{[}i,o,j{]} \in \Xi(o)}f^\text{h}_{ioj}  \nonumber \\
	& y_j^\text{a}(k) = \sum_{{[}i,o,j{]} \in \Xi(o)}f^\text{a}_{ioj}  \nonumber \\
	& y_j(k) = y_j^\text{h}(k) + y_j^\text{a}(k) 
\end{align}

To calculate the outgoing flows of the incoming cells,
\begin{align}\label{eq:outgoing_flow}
	\forall (i,o) \in \Gamma^{-1}(o) , \; & f_i^\text{h}(k) = \sum_{{[}i,o,j{]} \in \Xi(o)}f^\text{h}_{ioj}  \nonumber \\
	& f_i^\text{a}(k) = \sum_{{[}i,o,j{]} \in \Xi(o)}f^\text{a}_{ioj}  \nonumber \\
	& f_i(k) = f_i^\text{h}(k) + f_i^\text{a}(k) \; , 
\end{align}
where $\Gamma^{-1}(o)$ denotes the set of cells going into intersection $o$. \eqref{eq:dens_update} updates the human-driven and autonomous vehicle densities of each cell at the next time step. To update the fraction of vehicles in the outgoing cells on each path,
\begin{align*}
	& \forall (o,j) \in \Gamma(o) , \\
	& \mu_j^\text{h}(p,k+1) = \\
	& \; \frac{  \sum_{ {[}i',o,j{]} \in \Xi(o)}f^\text{h}_i(k) \mu^\text{h}_i(p,k) + \mu^\text{h}_j(p,k)(n^\text{h}_j(k) - f^\text{h}_j(k) )  }{ n^\text{h}_j(k+1)  } \; , \\
	& \mu_j^\text{a}(p,k+1) =   \\
	& \; \frac{  \sum_{ {[}i',o,j{]} \in \Xi(o)}f^\text{a}_i(k) \mu^\text{a}_i(p,k) + \mu^\text{a}_j(p,k)(n^\text{a}_j(k) - f^\text{a}_j(k) )  }{ n^\text{a}_j(k+1)  }  \; .
\end{align*}

\textbf{Accidents.} To evaluate the performance of the developed RL policy in reacting to disturbances, we consider stochastic accidents occuring in the network, each of which causes one lane to be closed. We let accidents occur in any cell at any time with equal probability as long as the jam density does not decrease below the current density of the cell. Each accident is cleared out after some number of time steps, drawn from a Poisson distribution. If $\bar{\numlanes}_{\cellidx}$ lanes of cell $\cellidx$ are closed due to accidents, then the jam density and the critical density for the cell reduce to $(\numlanes_{\cellidx} - \bar{\numlanes}_{\cellidx}) / \numlanes_{\cellidx}$ of their original values. Thus, accidents introduce time-dependency to these variables.

\section{Network dynamics: routing for humans and autonomous vehicles}
\label{sec:network_dynamics}
As mentioned above, we consider a network with a set of possible paths $\pathset$. We use $\demandhumRV$ and $\demandautRV$ to denote the human-driven and autonomous vehicle demands, respectively. We model all vehicles entering the network as entering a \emph{queue}, a single cell with infinite capacity. We use $0$ for the index of this cell. The routing choices of autonomous vehicles leaving the queue is determined by the central controller, and the routing choices of human-driven vehicles leaving it are determined from the latencies associated with each path, detailed below.

\subsection{Human choice dynamics}
\label{sec:human_choice_dynamics}

In general, people wish to minimize the amount of time spent traveling. However, people do not change routing choices instantaneously in response to new information; rather they have some inertia and only change strategies sporadically. Moreover, we assume people only account for current conditions and do not strategize based on predictions of the future~\cite{sandholm2010population}. Accordingly, we use an \emph{evolutionary dynamic} to describe how a population of users choose their routes.\footnote{Alternately, one could model individual users as learning agents, posing it as a Multi-Agent Reinforcement Learning problem. However, we consider large networks with too many human agents for this to be feasible.} Specifically, we model the human driver population as following Hedge Dynamics, also called Log-linear Learning~\cite{cesa2006prediction,marden2012revisiting,blume1993statistical}.

Let $(\routinghum_0(\pathidx, \timeidx))_{\pathidx\in\pathset}$ represent the initial routing of human-driven vehicles at time $k$; accordingly, $\sum_{\pathidx \in \pathset}\routinghum_0(p,k) = 1$ for all $k$. Humans will update their routes based on their estimates of how long it will take to traverse each path. However, it is not always possible to predict travel time accurately on general networks, since vehicles entering later on a different path may influence the travel time of vehicles entering earlier. Because of this, we consider that humans have an estimate $\hat{\ell}_{\pathidx}(\timeidx)$ of the true latency $\ell_{\pathidx}(\timeidx)$.
With these estimates, the routing vector is updated as follows.
\begin{equation}\label{eq:humanchoice}
\routinghum_{0}(\pathidx, \timeidx+1) = \frac{\routinghum_{0}(\pathidx, \timeidx)\exp(-\learningratehum(\timeidx)\hat{\latency}_{\pathidx}(\timeidx)) }{\sum_{\pathidx' \in \pathset}\routinghum_{0}(\pathidx', \timeidx)\exp(-\learningratehum(\timeidx)\hat{\latency}_{\pathidx'}(\timeidx)) } \; .
\end{equation}

The ratio of the volume of vehicles using a path at successive time steps is inversely proportional to the exponential of the delay experienced by users of that path. The learning rate $\learningratehum(\timeidx)$ may be decreasing or constant. Krichene \emph{et al.} introduce this model in the context of humans' routing choices and simulate a congestion game with Amazon Mechanical Turk users to show the model accurately predicts human behavior \cite{krichene2018learning}. We note that though we use this specific model for human choice for our simulations, the control method described later does not require this specific choice of human choice model. Our theoretical analysis similarly is not restricted to this choice of dynamics and works for any human choice model in which all fixed points of the dynamics satisfy human selfishness.

\vspace{-10px}
\subsection{Autonomous vehicle control policy}
\label{sct:control}

We assume that we have control over the routing of autonomous vehicles. We justify this by envisioning a future in which autonomous vehicles are offered as a service rather than a consumer product. We then assume that a city can coordinate with the owner of an autonomous fleet to decrease congestion in the city. Moreover, unlike traditional tolling, coordination between autonomous vehicles and city infrastructure allows for fast-changing and geographically finely quantized tolls, enabling routing control to be achieved through incentives~\cite{biyik2019green,beliaev2021incentivizing}. The initial routing of autonomous vehicles is then our control parameter by which we influence the state of traffic on the network. Consistent with the previous notation, we denote the initial autonomous routing as $(\routingaut_0(\pathidx, \timeidx))_{\pathidx \in \pathset} \in \mathbb{R}^{|\pathset|}_{\ge 0}$, where $\sum_{\pathidx \in \pathset}\routingaut_0(\pathidx,\timeidx) = 1$.

We assume the existence of a central controller, or social planner, which dictates $\routingaut_0$ by processing the state of the network. At each time step, we let the controller observe:
\begin{itemize}[nosep]
	\item the number of human-driven and autonomous vehicles in each cell and in the queue,
	\item binary states for each lane that indicates whether the lane is closed due to an accident or not.
\end{itemize}
We use deep RL to arrive at a policy for the social planner to control the autonomous vehicle routing, $\routingaut_0$. Since the state space is very large and both state and action spaces are continuous, a dynamic programming-based approach is infeasible. For instance, even if we discretized the spaces, say with $10$ quantization levels, and did not have accidents, we would have $10^{82}$ possible states and $10$ actions for a moderate-size network with only $2$ paths and $40$ cells in total.

We wish to minimize the total latency experienced by users, which is equal to summing over time the number of users in the system at each time step. Accordingly, the stage cost is:
\begin{align}\label{eq:cost_function}
J(\timeidx) = \sum_{\cellidx\in \cellset}\dens_{\cellidx}(\timeidx) \: .
\end{align}

Due to their high performance in continuous control tasks \cite{schulman2015trust,schulman2017proximal}, we employ policy gradient methods to learn a policy that produces $\routingaut_0$ given the observations. Specifically, we use state-of-the-art PPO with an objective function augmented by adding an entropy bonus for sufficient exploration \cite{schulman2017proximal,mnih2016asynchronous}. We build a deep neural network, and train it using Adam optimizer \cite{kingma2014adam}. An overview of the PPO method and the set of parameters we use are presented in the appendix (Sec.~\ref{sct:ppo_overview} and Sec.~\ref{sct:ppo_params}). Each episode has a fixed number of time steps. 

In order to evaluate the performance of our control policy, we use three criteria. The first is the \emph{throughput} of the network -- we wish to have a policy that can serve any feasible demand, thereby stabilizing the queue. The second is the \emph{average delay} experienced by users of the network, which we measure by counting the number of vehicles in the system. The third is the \emph{convergence} to some steady state; we wish to avoid wild oscillations in congestion. To contextualize the performance of our control policy in this framework, we first establish the performance of \emph{equilibria} of the network.

\vspace{-5px}
\section{Equilibrium analysis}
\label{sec:equilibria}

In this section, we examine the possible \emph{equilibria} of our dynamical system, which characterize the possible steady state behaviors of the system. A network with a given demand can have a variety of equilibria with varying average user delay. If our control achieves overall delay equal to that of the best possible equilibrium, it is a successful policy. Section~\ref{sec:results} shows empirically that our learned policy can achieve the best equilibrium in a variety of settings.

In this section, we first formulate an optimization which solves for the most efficient equilibrium, which is computationally hard. Motivated by this, we restrict the class of networks considered and prove theoretical properties of this restricted class. Using these properties, we formulate a new optimization formulation to solve for the most efficient equilibrium and prove it is solvable in polynomial time.

\vspace{-5px}
\subsection{Equilibrium Formulation}

We define two notions of equilibrium: one related to the vehicle flow dynamics, and one related to human choice dynamics.\footnote{These define equilibria in the sense of dynamical systems, and do not strictly correspond to game-theoretic notions of equilibria. Under Assumption \ref{assump:starting_flow} below, the set of equilibria for the dynamics of human choice will correspond to the set of Nash Equilibria where the payoff is the path latency.}

\begin{definition}[Path Equilibrium]
	We define a \emph{path equilibrium} for path $\pathidx$ as a set of cell densities $(\denshum_\cellidx(\timeidx),\densaut_\cellidx(\timeidx))_{\cellidx \in \cellset_\pathidx}$ that, for a given constant flow entering the first cell on the path, $\flowinhum_\cellidx(\timeidx)$ and $\flowinaut_\cellidx(\timeidx)$, the cell densities are constant.
\end{definition}

\begin{definition}[Network Equilibrium]\label{def:network_eq}
	We define a \emph{network equilibrium} as a set of cell densities $(\denshum_\cellidx(\timeidx), \densaut_\cellidx(\timeidx))_{\cellidx \in \cellset}$ and human vehicle routing $(\routinghum_0(\pathidx, \timeidx))_{\pathidx \in \pathset}$, such that for a given constant entering flow $\flowinhum_0(\timeidx)$ and $\flowinaut_0(\timeidx)$ and a given constant autonomous vehicle routing $(\routingaut_0(\pathidx, \timeidx))_{\pathidx \in \pathset}$, the human vehicle routing, subject to the dynamics in \eqref{eq:humanchoice}, is constant.
\end{definition}

We are interested in satisfying both notions of equilibrium -- both the path equilibrium, which deals with the vehicle flow dynamics, and the network equilibrium, which deals with the human choice dynamics. Accordingly, the pair can be considered a Stackelberg Equilibrium for a leader controlling the autonomous vehicles who wishes to maximize the social utility in the presence of selfish human demand. We formulate the following optimization to solve for the most efficient equilibrium (satisfying both notions of equilibrium defined above), \emph{i.e.} the equilibrium which minimizes the total travel time of all users of the network. We drop all time indices since we consider quantities that are constant over time.
\begin{flalign*}
\min_{(\denshum_\cellidx, \densaut_\cellidx, \flowhum_\cellidx, \flowaut_\cellidx, \flowinhum_\cellidx, \flowinaut_\cellidx, \routinghum_\cellidx(\pathidx), \routingaut_\cellidx(\pathidx), \ell_\pathidx)_{\cellidx \in \cellset, \pathidx \in \pathset}}\sum_{\cellidx \in \cellset} \dens_\cellidx &&
\end{flalign*}
\begin{align*}
	\text{s.t. } \forall o \in O: & \text{ \textbf{procedure} \textsc{Flow Calculation}(Intersection $o$)} \\
	& \eqref{eq:incoming_flow}, \eqref{eq:outgoing_flow} \\
	\forall \cellidx \in \cellset: \; & \flowinhum_\cellidx = \flowhum_\cellidx, \quad \flowinaut_\cellidx = \flowaut_\cellidx, \quad \autlev_\cellidx = \densaut_\cellidx/(\denshum_\cellidx + \densaut_\cellidx) \\
	& \sum_{\cellidx' \in \upstreamcells_\cellidx}\! \flowhum_{\cellidx'}\routinghum_{\cellidx'}(\pathidx) \!=\! \flowhum_\cellidx \routinghum_\cellidx(\pathidx), \: \sum_{\cellidx' \in \upstreamcells_\cellidx}\! \flowaut_{\cellidx'}\routingaut_{\cellidx'}(\pathidx) \!=\! \flowaut_\cellidx \routingaut_\cellidx(\pathidx) \\
	& \ell_\pathidx = \sum_{\cellidx \in \pathidx} (\denshum_\cellidx + \densaut_\cellidx)/(\flowhum_\cellidx + \flowaut_\cellidx) \\
	\forall \pathidx,\pathidx' \in \pathset: \; & \routinghum_0(\pathidx)(\ell_\pathidx - \ell_{\pathidx'}) \le 0
\end{align*}

While this formulation solves for the most efficient equilibrium of any traffic network, it is computationally difficult, especially due to the final constraint. Due to this, we introduce a restricted class of networks that we consider for the remainder of this section, which allows us to compute equilibria in polynomial time with respect to the number of paths.
\begin{definition}[Bottleneck]
	We define a \emph{bottleneck} as a regular junction at which the number of lanes decreases, decreasing the capacity of the cells.
\end{definition}

\begin{assumption}\label{asmp:parallel}
	We consider a parallel network in which leaving the first cell, vehicles choose a path and paths do not share cells, meaning that each cell is identified with only one path, aside from the downstream-most cell which has infinite capacity. We further consider that all cells in the path have the same model parameters, except for a bottleneck after the $\numnbcells$ upstream-most cells.
\end{assumption}
In other words, we consider a parallel network where each path is composed of identical cells except for a single junction with a decrease in the number of lanes. Fig.~\ref{fig:schematic} shows an example of such a network. For ease of analysis, we first establish properties of Path Equilibria, then Network Equilibria.

\subsection{Path equilibrium}
\label{sct:path_eq}

As mentioned above, we restrict our considered class of paths to those with a single bottleneck, meaning one point on the path at which cell capacity drops. Formally, we consider each path $\pathidx$ to have $\numnbcells$ cells, each with $\numlanesnb_\pathidx$ lanes, followed by $\numbcells$ cells downstream, each with $\numlanesb_\pathidx$ lanes, where $\numlanesb_\pathidx < \numlanesnb_\pathidx$. We define $\laneratio_\pathidx:=\numlanesb_\pathidx/\numlanesnb_\pathidx \in(0,1)$. 

In a slight abuse of notation, we use the subscript $\pathidx$ for parameters that are constant over a path under Assumption \ref{asmp:parallel}, and the superscript $\text{n}$ for cells before the bottleneck and $\text{b}$ for the bottleneck and cells downstream of it. We now present a theoretical result that completely analytically characterizes the path latencies that can occur at equilibrium.

\begin{theorem}\label{thm:road_equilibrium}
	Under Assumption~\ref{asmp:parallel} a path $\pathidx$ with flow dynamics described in Section~\ref{sec:flow_dynamics} that is at Path Equilibrium will have the same autonomy level in all cells. Denote this autonomy level $\autlev_{\pathidx}$. If the vehicle flow demand is strictly less than the minimum cell capacity, the path will have no congested cells. Otherwise, the path will have one of the following latencies, where $\gamma_\pathidx \in \{0, 1, 2, \ldots, \numnbcells\}$:
	\begin{align*}
		\latency_{\pathidx} = \frac{|\cellset_\pathidx|}{\ffvelocity_\pathidx} + \gamma_\pathidx \frac{(1-\laneratio_\pathidx) \jamdennb_\pathidx (\autlev_{\pathidx} \spaceaut_\pathidx + (1-\autlev_{\pathidx}) \spacehuman_\pathidx)}{\laneratio_\pathidx  \ffvelocity_\pathidx \numlanesnb_\pathidx} \; .
	\end{align*}
\end{theorem}
\begin{proof}
	The proof is composed of three lemmas. We first establish a property of path equilibria that allows us to treat the vehicle flow as if it were composed of a single car type. With this, we use the CTM to characterize possible equilibria on a path. We then derive the delay associated with each congested cell. Combining the latter two lemmas yields the theorem. 
	
	\begin{lemma}\label{lma:homogeneousautlevel}
		A path in equilibrium with nonzero incoming flow has the same autonomy level in all cells of the path, which is equal to the autonomy level of the vehicle flow onto the path. Formally, a path $\pathidx$ with demand $(\demandhum_{\pathidx}, \demandaut_{\pathidx})$ in equilibrium has, for all cells $\cellidx$ in $\cellset_\pathidx$,
		\begin{equation*}
		\autlev_{\cellidx} = \demandaut_{\pathidx}/(\demandhum_{\pathidx} + \demandaut_{\pathidx}) \; .
		\end{equation*}
	\end{lemma}
	We defer the proof of the lemma to the appendix. With this, our path equilibria analysis simplifies to that of single-typed traffic, with the autonomy level treated as a variable parameter. The next lemma, similarly to Theorem 4.1 of \cite{gomes2008behavior}, completely characterizes the congestion patterns that can occur in cell equilibria. For this lemma, we consider the cell indices in a path to be increasing, where the cell immediately downstream from a cell $\cellidx$ has index $\cellidx\!+\!1$.
	
	\begin{lemma}\label{lma:congestion_profiles}
		Under Assumption \ref{asmp:parallel}, if the demand on a path is less than the minimum capacity of its cells, they will be uncongested at path equilibrium. Otherwise, a path with demand equal to the minimum cell capacity will have $\numnbcells$ possible path equilibria, corresponding to one of the following sets of congested cells, where $j$ is the index of the $\numnbcells$th cell:
		\begin{align*}
			\left\{\emptyset, \{j\}, \{j-1, j \}, \ldots, \{ j - \numnbcells +1, \ldots, j-2, j-1, j \} \right\} .
		\end{align*}
		\begin{proof}
			As mentioned above, this lemma relates closely to Theorem~4.1 of \cite{gomes2008behavior}. However, we cannot directly apply that theorem due to differing assumptions; namely they assume $\capacity_{\cellidx+1}=(\jamden_\cellidx - \critdens_\cellidx) \shockspd_\cellidx$ for all $\cellidx$. We therefore offer a similar proof, tailored to our assumptions.
			
			For ease of notation, we drop all path subscripts $\pathidx$ as well as the cell index for the free-flow velocity parameter $\ffvelocity$. In light of Lemma~\ref{lma:homogeneousautlevel}, we also suppress the autonomy level arguments to capacity $\capacity_\cellidx$ and critical density $\critdens_\cellidx$. The flow equation then becomes $\flow_\cellidx = \min(\ffvelocity\dens_\cellidx, \; (\jamden_{\cellidx+1} - \dens_{\cellidx + 1})\shockspd_{\cellidx + 1}, \; \capacity_\cellidx, \capacity_{\cellidx+1} )$.
			
			We begin by proving that if the vehicle flow demand is strictly less than the minimum capacity, \emph{i.e.} the bottleneck capacity, then the only equilibrium has no congested cells. Let us use $j'$ to denote the index of the final cell in the path. Under Assumption \ref{asmp:parallel} there is no supply limit to the flow exiting a path, so $\flow_{j'} = \min(\ffvelocity \dens_{j'}, \capacity_{j'})$. Since $\flow_0 = \flow_{j'} < \capacity_{j'}$, $\flow_0 = \flow_{j'} = \ffvelocity \dens_{j'}$. The definition of capacity, $\capacity_\cellidx = \ffvelocity \critdens_\cellidx$, then implies that $\dens_{j'} < \critdens_{j'}$, meaning that cell ${j'}$ is uncongested, so $\ffvelocity \dens_{j'} < (\jamden_{j'} - \dens_{j'})\shockspd_{j'}$.
			
			This is the base case for a proof by induction. Consider cell $\cellidx$ that is uncongested (\emph{i.e.} $\dens_\cellidx < \critdens_\cellidx$). Since by assumption all cells have flow strictly less than the cell's capacity, $\flow_\cellidx = \ffvelocity \dens_\cellidx < \capacity_\cellidx$. Then consider the flow entering cell $\cellidx$: $\flow_{\cellidx-1} \!=\! \min(\ffvelocity \dens_{\cellidx - 1}, (\jamden_\cellidx - \dens_\cellidx)\shockspd_\cellidx, \capacity_{\cellidx-1}) \!=\! \flow_{\cellidx} \!<\! \capacity_\cellidx \!<\! (\jamden_\cellidx \!-\! \dens_\cellidx)\shockspd_\cellidx$.
			
			The fact that $\capacity_\cellidx \le \capacity_{\cellidx - 1}$ then implies that $\flow_{\cellidx-1} = \ffvelocity \dens_{\cellidx - 1}$, so cell $\cellidx-1$ is uncongested, proving the lemma's first statement.
			
			The second statement assumes the flow on the path is equal to the minimum capacity. The cells in the bottleneck segment all have the same capacity, which we denote $\capacity^\text{b}$; this capacity is less than the capacity of the cells in the nonbottleneck segment. This means all bottleneck cells will be operating at capacity (and therefore have vehicle density equal to their critical density); flow on the path is therefore equal to $\capacity^\text{b}$.
			
			We now turn to the nonbottleneck segment. We first note that if a nonbottleneck cell is uncongested then the preceeding cell must be uncongested as well, using the same reasoning as that proving the first statement above. Next, consider the flow out of the downstream-most cell of the nonbottleneck segment: $\flow_{j} \!=\! \min(\ffvelocity \dens_{j}, (\jamden_{j+1}\! -\! \dens_{j+1}) \shockspd_{j+1}, \capacity_{j}) \!=\! \capacity^\text{b} \!<\! \capacity_{j}$, so $\flow_{j} \!=\! \min(\ffvelocity \dens_{j}, (\jamden_{j+1} \!-\! \dens_{j+1}) \shockspd_{j+1})$.
			Cell $j$ can be uncongested, in which case the cell density is such that $\ffvelocity \dens_{j} = \capacity^\text{b}$, or the cell can be congested, in which case the second term dominates. Then, if nonbottleneck cell $\cellidx$ is congested, the flow into it is  $\flow_{\cellidx-1} \!=\! \min(\ffvelocity \dens_{\cellidx-1}, (\jamden_{\cellidx} \!-\! \dens_{\cellidx}) \shockspd_{\cellidx})$.
			Again, to achieve this flow, cell $\cellidx\!-\!1$ can be either congested or uncongested. As shown above, if uncongested, then all upstream cells must be uncongested as well, yielding the second statement in the lemma.	
		\end{proof}
	\end{lemma}
	
	We use these properties to find a closed-form expression for the latency incurred by traveling through a bottleneck cell, which when combined with Lemma \ref{lma:congestion_profiles}, completes the proof.
	
	\begin{lemma}\label{lma:congested_latency}
		The latency incurred by traveling through a congested cell is as follows.
		\begin{align*}
		\frac{1}{\ffvelocity_\pathidx} + \frac{(1-\laneratio_\pathidx)\jamdennb_\pathidx(\autlev_{\pathidx}\spaceaut_\pathidx + (1-\autlev_{\pathidx})\spacehuman_\pathidx)}{\laneratio_\pathidx \ffvelocity_\pathidx\numlanesnb_\pathidx} \; .
		\end{align*}
		\begin{proof}\phantom{\qedhere}
			Recall that we assume paths have a uniform free-flow velocity across all cells in a path, where path $\pathidx$ has free-flow velocity $\ffvelocity_\pathidx$. We define $\nbcellsset$ as the set of cells before the bottleneck, which have $\numlanesnb_\pathidx$ lanes. The remaining cells, with indices in the set $\cellset_\pathidx \setminus \nbcellsset$, have $\numlanesb_\pathidx$ lanes. Further recall the definition $\laneratio_\pathidx = \numlanesb_\pathidx / \numlanesnb_\pathidx$. Let $\capacitynb_\pathidx(\autlev_{\pathidx})$ denote the capacity of the cells before the bottleneck of path $\pathidx$ with autonomy level $\autlev_{\pathidx}$ and let $\capacityb_\pathidx(\autlev_{\pathidx})$ be the same for the bottleneck cell. Note that $\capacityb_\pathidx(\autlev_{\pathidx}) = \laneratio_{\pathidx} \capacitynb_\pathidx(\autlev_{\pathidx})$. Similarly, let $\shockspdnb_\pathidx(\autlev_{\pathidx})$ and $\shockspdb_\pathidx(\autlev_{\pathidx})$ denote the shockwave speed for prebottleneck cells and bottleneck cell, respectively, on path $\pathidx$ with autonomy level $\autlev_{\pathidx}$, as with jam densities $\jamdennb_\pathidx$ and $\jamdenb_\pathidx$ and critical densities $\critdensnb_\pathidx(\autlev_{\pathidx})$ and $\critdensb_\pathidx(\autlev_{\pathidx})$.

			Lemma~\ref{lma:congestion_profiles} establishes all possible combinations of congested cells that a path at equilibrium can experience. We now investigate how much delay each configuration induces on the path, parameterized by the autonomy level of the path. By Lemma~\ref{lma:congestion_profiles} and the definitions of $\laneratio$ and capacity \eqref{eq:capacity},
			\begin{equation}\label{eq:flow_into_bottleneck}
			\flow_\pathidx=\capacityb_\pathidx = \laneratio_\pathidx \capacitynb_\pathidx(\autlev_{\pathidx}) = \laneratio_\pathidx \shockspdnb_\pathidx(\autlev_{\pathidx}) (\jamdennb_\pathidx - \critdensnb_\pathidx(\autlev_{\pathidx}))  \; .
			\end{equation}
			
			Let $\denscong_\pathidx(\autlev_{\pathidx})$ denote the vehicle density in a congested cell on path $\pathidx$, which we know must occur upstream of the bottleneck (Lemma~\ref{lma:congestion_profiles}). Then, the flow entering a congested cell before the bottleneck is $\flow_\pathidx = \shockspdnb_\pathidx(\autlev_{\pathidx}) (\jamdennb_\pathidx - \denscong_\pathidx(\autlev_{\pathidx}))$. Equating this with \eqref{eq:flow_into_bottleneck}, we find $\denscong_\pathidx(\autlev_{\pathidx}) = (1-\laneratio_\pathidx) \jamdennb_\pathidx + \laneratio_\pathidx  \critdensnb_\pathidx(\autlev_{\pathidx})$.
			
			To use this to find the latency incurred by traveling through a congested cell, we divide the density by the flow, as follows.
			\begin{align*}
			\frac{\denscong_\pathidx(\autlev_{\pathidx})}{\flow_{\pathidx}} &= \frac{\denscong_\pathidx(\autlev_{\pathidx})}{\capacityb_\pathidx(\autlev_{\pathidx})} = \frac{(1-\laneratio_\pathidx) \jamdennb_\pathidx + \laneratio_\pathidx \critdensnb_\pathidx(\autlev_{\pathidx})}{\laneratio_\pathidx \ffvelocity_\pathidx \critdensnb_\pathidx(\autlev_{\pathidx})}  \\
			&= \frac{1}{\ffvelocity_\pathidx} + \frac{(1-\laneratio_\pathidx) \jamdennb_\pathidx (\autlev_{\pathidx} \spaceaut_\pathidx + (1-\autlev_{\pathidx}) \spacehuman_\pathidx)}{\laneratio_\pathidx  \ffvelocity_\pathidx \numlanesnb_\pathidx} \; .\qquad\qed
			\end{align*}
		\end{proof}
	\end{lemma}
Together, the lemmas prove the theorem.
\end{proof}

The two terms above are the free-flow delay and the per-cell latency due to congestion, respectively. Theorem~\ref{thm:road_equilibrium} allows us to calculate the possible latencies of a path as a function of its autonomy level $\autlev_{\pathidx}$. Since in a network equilibrium all used paths have the same latency, we can calculate network equilibria more efficiently than comprehensively searching over all possible routings. However, equilibria may not exist, even with a fine time discretization -- in equilibrium the path latencies must be equal, but by Theorem \ref{thm:road_equilibrium}, road latency is a function of the integer $\gamma_\pathidx$. To avoid this artifact, when analyzing network equilibria we consider the cells to be small enough that we can consider the continuous variable $\gamma_\pathidx \in [0,\numnbcells]$.

\vspace{-10px}
\subsection{Network equilibrium}
\label{sct:network_eq}

We define the best equilibrium to be the equilibrium that serves a given flow demand with minimum latency. We are now ready to establish properties of network equilibria, as well as how to compute the best equilibria. We use the following two assumptions in our analysis of network equilibrium.

\begin{assumption}\label{assump:ff_latency}
	No two paths have the same free-flow latency.
\end{assumption}

\begin{assumption}\label{assump:starting_flow}
	The initial choice distribution has positive human-driven and autonomous vehicle flow on each path.
\end{assumption}

Note that the Assumption \ref{assump:ff_latency} is not strictly necessary but is useful for easing analysis. A similar analysis could be performed in its absence. We justify Assumption \ref{assump:starting_flow} by noting that humans are not entirely rational and that our choice model does not capture all reasons a person may wish to choose a route, and some small fraction of people will choose routes that seem less advantageous at first glance.

\begin{theorem}\label{thm:network_eq}
	Under Assumptions \ref{asmp:parallel} and \ref{assump:ff_latency}, a routing that minimizes total latency when all users (both human drivers and autonomous users) are selfish can be computed in $O(|\pathset|^3\log|\pathset|)$ time. A routing that minimizes total latency when human drivers are selfish and autonomous users are controlled can also be computed in $O(|\pathset|^3\log|\pathset|)$ time.
\end{theorem}
\begin{proof}
	
	To establish properties of network equilibria, we introduce some notation. We use $\ffl_\pathidx = |\cellset_\pathidx|/\ffvelocity_\pathidx$ to denote the free-flow latency of path $\pathidx$. We also use $\pathset_{\le \ffl_{\pathidx}} = \{ \pathidx' \in \pathset : \ffl_{\pathidx'} \le \ffl_{\pathidx} \}$, which denotes the set of paths with free-flow latency less than or equal to that of path $\pathidx$. We similarly define the expression with other comparators, \emph{e.g.} $\pathset_{< \ffl_{\pathidx}}$ or $\pathset_{> \ffl_{\pathidx}}$.
	
	This proposition follows from Definition \ref{def:network_eq} and Assumption~\ref{assump:starting_flow}. The next lemma follows, with proof deferred to the appendix.
\begin{proposition}\label{prop:equilibrium}
	In a network equilibrium,
	\begin{enumerate}[nosep]
		\item All paths with selfish drivers have the same latency, and
		\item All paths without selfish drivers have equal or greater latency.
	\end{enumerate}
\end{proposition}

\begin{lemma}\label{lma:ff}
	If the set of equilibria contains a routing with positive flow only on paths $\pathset_{\le \ffl_{\pathidx}} $, then there exists a routing in the set of equilibria in which path $\pathidx$ is in free-flow.
\end{lemma}
	
	\begin{lemma}\label{lma:best_eq}
		Under Assumption~\ref{assump:ff_latency}, if some users are selfish and some users are not selfish, then the best equilibrium will have the following properties:
		\begin{enumerate}[nosep]
			\item the path with largest free-flow latency used by selfish users will be in free-flow,
			\item all paths with lower free-flow latency will be congested,
			\item paths with greater free-flow latency may have nonselfish users, and
			\item paths used with larger free-flow latency that have nonselfish users on them will be at capacity, except perhaps the path with largest free-flow latency used by nonselfish users.
		\end{enumerate}
	\end{lemma}

\begin{proof}
	Consider a network with some selfish and some non-selfish (controlled) users. Let $\pathidx$ denote the path with the longest free-flow latency that contains selfish users. For the purpose of contradiction, let this path contain congested cells, and let this be the best equilibrium. Fix the nonselfish flow on all roads with longer free-flow latency than $\pathidx$. By Lemma \ref{lma:ff}, there exists an equilibrium for the selfish users in which $\pathidx$ is in free-flow. This results in less latency for the users on path $\pathidx$, and no selfish user will have greater delay (Proposition \ref{prop:equilibrium}). This contradicts the premise, proving the first property.
	
	The second property follows directly from Proposition~\ref{prop:equilibrium} and Assumption \ref{assump:ff_latency}. The third property follows from the definition of nonselfish users, which can take a path with a larger latency than other available paths. The best equilibrium minimized total latency. If there was a road with nonselfish users that was not at capacity, while another path with higher latency has positive flow, this would not be the best equilibrium, since a more efficient routing would shift flow from the higher latency path to the lower latency one. This yields the final property.
\end{proof}

Using these properties, we prove Theorem \ref{thm:network_eq}. We first consider the setting in which all users are selfish. We use $\latencycong_\pathidx(\autlev_{\pathidx})$ to denote the per-cell latency due to congestion, \emph{i.e.} $\latencycong_\pathidx(\autlev_{\pathidx}) = \frac{(1-\laneratio_\pathidx)\jamdennb_\pathidx(\autlev_{\pathidx}\spaceaut_\pathidx + (1-\autlev_{\pathidx})\spacehuman_\pathidx)}{\laneratio_\pathidx \ffvelocity_\pathidx\numlanesnb_\pathidx}$. Lemma~\ref{lma:best_eq} implies that for a given demand, all equilibria in the set of most efficient equilibria for that demand have one path that is in free-flow. We can then formulate the search for a best equilibrium as an optimization. We are helped by the fact that the best equilibria will use the minimum number of feasible paths, since all users experience the same delay. Then, for each candidate free-flow path (denote with index $p'$), check feasibility of only using paths $\pathset_{\le \ffl_{\pathidx'}}$, and choose a routing that minimizes $|\pathset_{\le \ffl_{\pathidx'}}|$, \emph{i.e.} the number of roads used. The reason for minimizing the number of used roads is that all users are experiencing the same latency (Proposition \ref{prop:equilibrium}) and in the best equilibrium, the road with flow on it that has longest free-flow latency will be in free-flow (Lemma \ref{lma:best_eq}). The feasibility can be checked as follows, with an optimization that utilizes Lemma~\ref{lma:best_eq}.
\begin{flalign*}
\noalign{$\displaystyle\argmin_{ (\flowhum_\pathidx, \flowaut_\pathidx)_{\pathidx \in \pathset_{\le \ffl_{\pathidx'}}} , \; \gamma \in \prod_{ \pathidx \in \pathset_{< \ffl_{\pathidx'}} }[0,\numnbcells] } 1$} &&
\end{flalign*}
\vspace{-1cm}
\begin{align*}
	\text{s.t. } & \sum_{\pathidx \in  \pathset_{\le \ffl_{\pathidx'}} } \flowhum_\pathidx = \demandhum, \sum_{\pathidx \in \pathset_{\le \ffl_{\pathidx'}} } \flowaut_\pathidx = \demandaut \\
	&\flowhum_{p'} + \flowaut_{p'} \le \capacity_{p'}(\frac{\flowaut_{p'}}{\flowhum_{p'} + \flowaut_{p'}}) \\
	\forall \pathidx \in \pathset_{< \ffl_{\pathidx'}}: \; &\gamma_\pathidx \latencycong_\pathidx(\frac{\flowaut_\pathidx}{\flowhum_\pathidx + \flowaut_\pathidx}) = \ffl_{\pathidx'} - \ffl_{\pathidx}  \\
	&\flowhum_\pathidx + \flowaut_\pathidx = \capacity_\pathidx(\frac{\flowaut_\pathidx}{\flowhum_\pathidx + \flowaut_\pathidx}) 
\end{align*}
The last constraint yields an affine relationship between $\flowhum_\pathidx$ and $\flowaut_\pathidx$ for paths $\pathset_{< \ffl_{\pathidx'}}$. Solving for $\flowhum_\pathidx$ and plugging into the first constraint yields an affine relationship between $\gamma_{\pathidx}$ and $\flowaut_\pathidx$. This way, the optimization can be converted to a linear program, and we must solve $\log|\pathset|$ linear programs to search the minimum feasible $p'$.
	
	This formulation assumes that all vehicles are selfish. If instead we consider selfish human drivers and fully controlled autonomous users, we can construct a similar optimization to find the best equilibrium.  For each choice of free-flow path $p'$, we minimize the total latency of the autonomous vehicles not on free-flow paths. We then choose the routing corresponding to free-flow path $p'$ that minimizes total latency (which may not necessarily minimize the number of paths used by human drivers). For each candidate free-flow path $\pathidx'$ we solve the following optimization.
\begin{flalign*}
	\noalign{$\displaystyle\argmin_{ (\flowhum_\pathidx, \flowaut_\pathidx)_{\pathidx \in \pathset}, \; \gamma \in \prod_{\pathidx \in \pathset_{< \ffl_{\pathidx'}} }[0,\numnbcells] }\sum_{\pathidx \in \pathset_{> \ffl_{\pathidx'}} }\flowaut_{\pathidx}\ffl_\pathidx $} &&
\end{flalign*}
\vspace{-1cm}
\begin{align*}
	\text{s.t. } & \sum_{\pathidx \in \pathset_{\le \ffl_{\pathidx'}} } \flowhum_\pathidx = \demandhum, \sum_{\pathidx \in \pathset} \flowaut_\pathidx = \demandaut\nonumber\\
	&\flowhum_{p'} + \flowaut_{p'} \le \capacity_{p'}(\frac{\flowaut_{p'}}{\flowhum_{p'} + \flowaut_{p'}})\\
	\forall \pathidx \in \pathset_{< \ffl_{\pathidx'} }: \; &\gamma_\pathidx \latencycong_\pathidx(\autlev_{\pathidx}) = \ffl_{\pathidx'} - \ffl_{\pathidx} \\
	&\flowhum_\pathidx + \flowaut_\pathidx = \capacity_\pathidx(\frac{\flowaut_\pathidx}{\flowhum_\pathidx + \flowaut_\pathidx}) \\
	\forall \pathidx \in \pathset_{> \ffl_{\pathidx'}} : \;&\flowaut_\pathidx \le \capacityb(1)
\end{align*}
	This can be reformulated as a linear program by the same mechanism. Again, we solve $\log|\pathset|$ linear programs and choose the one corresponding to the minimum feasible $\pathidx'$.
\end{proof}
Using these properties to compute optimal equilibria, we establish a framework for understanding the performance of our learned control policy. If the policy can reach the best equilibrium latency starting from arbitrary path conditions we view the policy as successful. We use this baseline to evaluate our experimental results in the following section.

A question then arises: if we have computed the best possible equilibria, why do we not directly implement that control? This approach is not fruitful, since the theoretical analysis of best equilibria gives the control policy only in the steady state. In practice, the network can start in any state, including worse equilibria, from which good equilibria will not emerge when autonomous vehicles unilaterally use their routing in the best equilibrium. Besides, our equilibrium analysis is limited to parallel networks and extending it to more general networks would yield a nonconvex optimization problem. A dynamic policy which depends on the current traffic state is therefore needed to guide the network to the best equilibrium. As shown in the following section, the policy learned via deep reinforcement learning achieves this guidance and reaches the best equilibrium in a variety of settings.

\section{Experiments and results}
\label{sec:results}
In all of the experiments\footnote{We make the code available in the supplementary material.}, we adopt the following parameters. All vehicles are $4$ meters long. Human drivers keep a $2$ second headway distance, whereas autonomous cars can keep $1$ second. Each time step corresponds to $1$ minute of real-life simulation. Each episode during deep RL training covers $5$ hours of real-life simulation ($300$ time steps). In test time, we simulate $6$ hours of real-life ($360$ time steps) to ensure the RL policy did not learn to minimize the latency in the first $300$ time steps and leave excess vehicles in the network at the end. We divide paths into the cells such that it takes $1$ time step to traverse each cell in free-flow. We initialize $\dens_{\cellidx}(0) \sim \textrm{unif}(0,1.2\critdens_{\cellidx})$ for all $\cellidx \in \cellset_\pathidx$ for all $\pathidx \in \pathset$. We set the standard deviations of the zero-mean Gaussian demand noise to be $\demandhum/10$ and $\demandaut/10$ for human-driven and autonomous vehicles, respectively.

Our overall control scheme can be seen in Fig.~\ref{fig:schematic}. As the learning model, we build a two-hidden-layer neural network, with each layer having $256$ nodes. We train an RL agent for each configuration that we will describe later on in simulated traffic networks that are based on the mixed-autonomy traffic model and the dynamics that we described in Sections~\ref{sec:flow_dynamics} and \ref{sec:network_dynamics}. All trainings simulate $40$ million time steps.\footnote{Other hyperparameter values we use for PPO are in the Appendix.} Depending on whether we evaluate our RL-based approach with (or without) the accidents, we enable (or disable) accidents at the training phase. However, we note that the number of possible accident configurations in the network is far more than the expected number of accidents during all training episodes. Hence, successfully handling accidents requires good generalization performance. Similar to accidents, the demand distributions match between training and test environments.

We compare our method with two baselines: first, a \emph{selfish} routing scheme, where all cars are selfish and use the human choice dynamics presented in Sec.~\ref{sec:human_choice_dynamics}, and second, a model predictive control (\emph{MPC}) based controller which can perfectly simulate the network other than the uncertainty due to accidents and noisy demand. It plans for the receding horizon of $4$ minutes and re-plans after every $1$ minute to minimize the number of cars in the network using a Quasi-Newton method (L-BFGS \cite{andrew2007scalable}). To increase robustness against the uncertainty, it samples $12$ different simulations of the network and takes the average. We note that this MPC can only be useful in small networks where some cars can enter the network and reach the destination within the MPC horizon of $4$ minutes. While increasing the horizon may help MPC operate in larger networks, it causes a huge computational burden. In fact, even though we parallelized the controller over $12$ Intel$^{\textrm{\textregistered}}$ Xeon$^{\textrm{\textregistered}}$ Gold 6244 CPUs (3.60 GHz), it took the controller $32$ seconds on average to decide the routing of autonomous vehicles for the next $1$ minute, which clearly indicates a practical problem. In all experiments, we set $\learningratehum(\timeidx)$ (and $\learningrateaut(\timeidx)$ for the selfish baseline) to be $0.5$ for all $k$.

\begin{figure}[t]
	\centering
	\includegraphics[width=\linewidth]{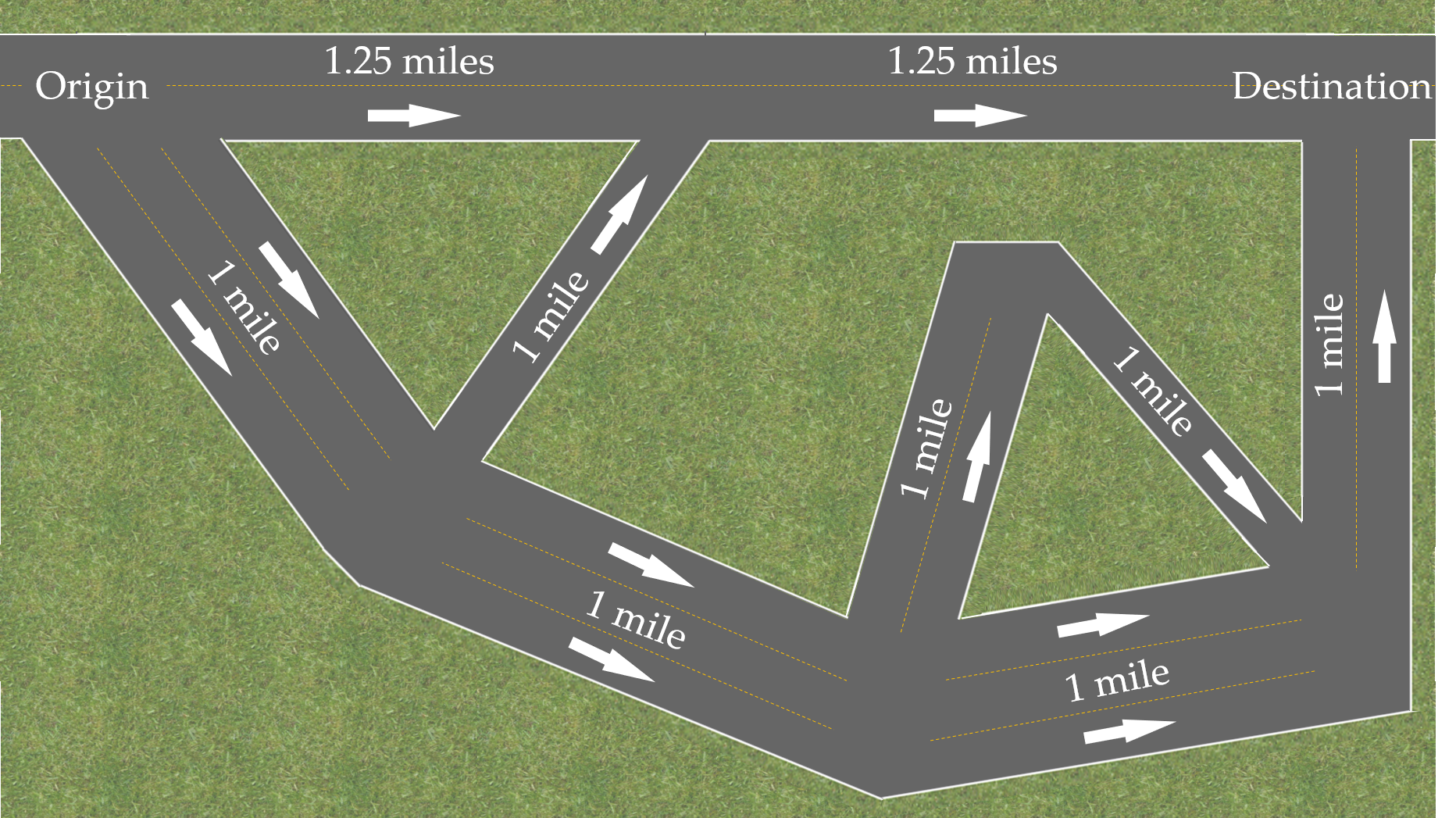}
	\caption{The small general class network used for experiments.}
	\label{fig:general_network}
\end{figure}

\subsection{General Class of Networks} We first start by considering a small network of $9$ cells and $7$ junctions ($1$ regular junction, $3$ merges and $3$ diverges) as shown in Fig.~\ref{fig:general_network}, where the priority levels of cells at merges are equal to their numbers of lanes. We set the autonomy level of the demand $\bar{\autlev}\!=\!0.6$ and the total demand $\demandhum + \demandaut = 2.60$ cars per second. We set the probability of accidents such that the expected frequency of accidents is $1$ per $100$ minutes, and clearing out an accident takes $30$ minutes on average \cite{houstontranstar2018}. For human choice dynamics, we assume humans' latency estimates are based on the current states of each cell, i.e., they estimate the latencies as if the network is in steady-state.

\begin{figure}[h]
	\centering
	\includegraphics[width=\linewidth]{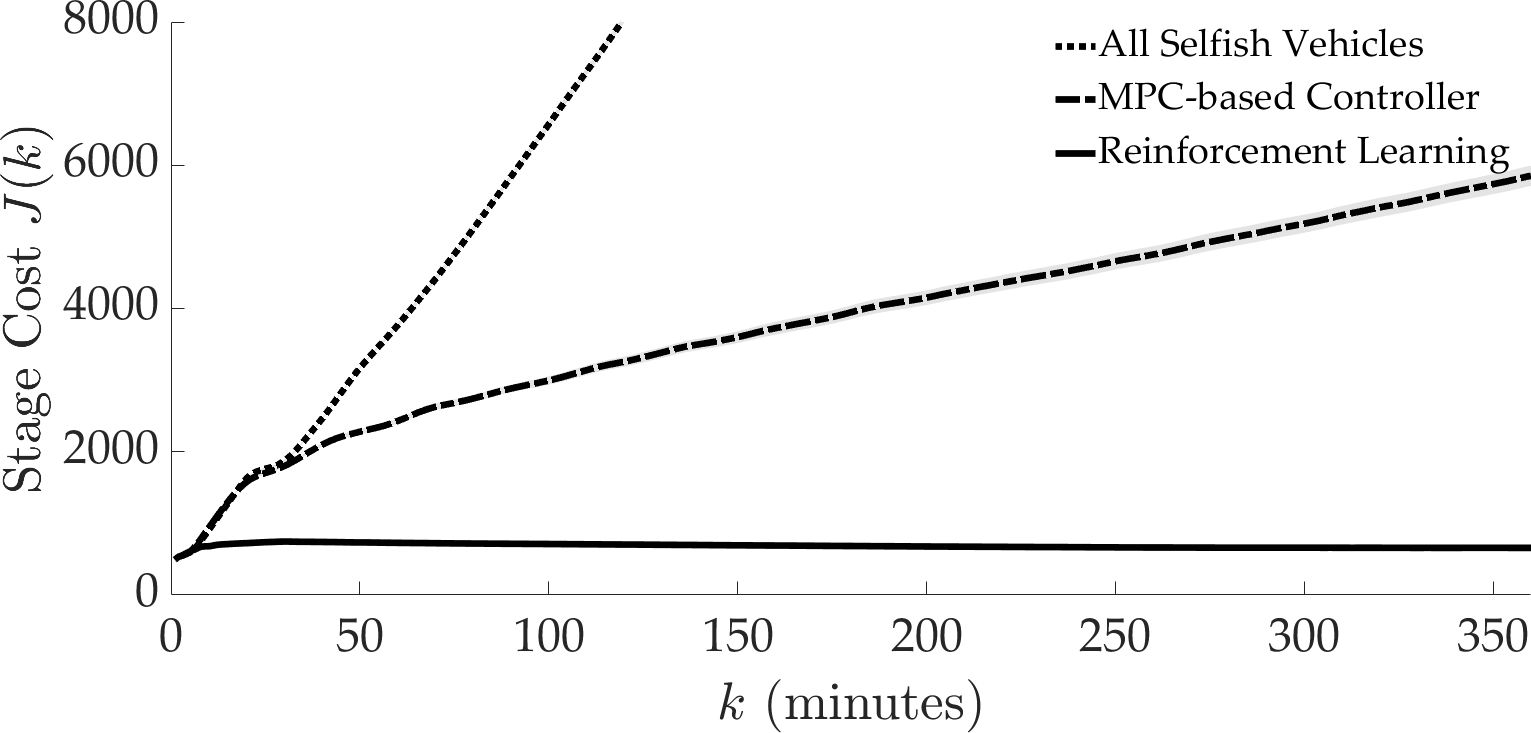}
	\caption{Time vs. number of cars under selfish, MPC and RL routing on the small general class network.}
	\label{fig:rl_vs_mpc}
\end{figure}

Fig.~\ref{fig:rl_vs_mpc} shows the number of cars in the network over time (mean $\pm$ standard error over $100$ simulations). While MPC controller improves over the selfish routing, they both suffer from linearly growing queues. On the other hand, RL controller stabilizes the queue and keeps the network uncongested.

\begin{figure}[h]
	\centering
	\includegraphics[width=\linewidth]{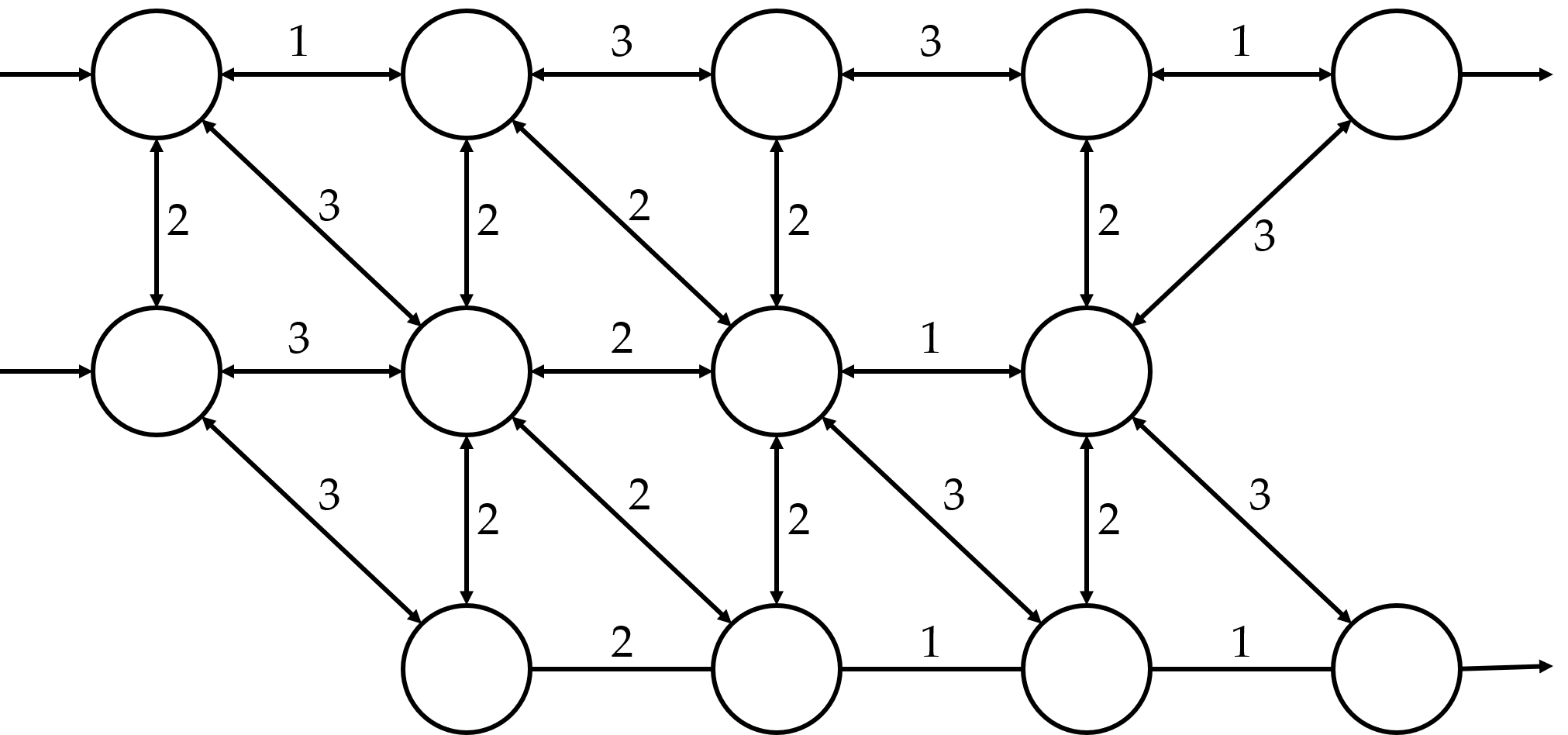}
	\caption{OW network (adapted from \cite{de2011modelling}) used for experiments.}
	\label{fig:ow_network}
\end{figure}

Next, we consider a larger network shown in Fig.~\ref{fig:ow_network} as a graph where the numbers noted on the links denote the number of cells in that link in one direction. Each cell, excluding queues which has infinite capacity, has $2$ lanes. This is a quantized version of the OW network due to Ortúzar and Willumsen \cite{de2011modelling}, and is widely used in the transportation literature \cite{ramos2015towards,bazzan2016multiagent,bazzan2014evolutionary,grunitzki2014individual}. This is a larger network with $4$ origin-destination pairs, $102$ cells (and $2$ queues) and $41$ junctions ($28$ junctions with only one incoming and one outgoing cell, and $13$ more general junctions). We set the total demand to be $\demandhum + \demandaut = 3.46$ cars per second, distributed equally to the $4$ origin-destination pairs in expectation. As there are $1752$ possible different simple paths that vehicles could be taking, our action space is $1752$ dimensional. While such an optimization is still possible with powerful computation resources, it might be unnecessary because an optimal solution is unlikely to utilize the paths that traverse too many cells. We therefore restrict our action space to the $10$ shortest paths (with respect to the free-flow latencies) between each origin and destination, and so adopt a $40$-dimensional action space. We keep the other experiment parameters the same as the small network experiment above.

Due to the network size and the computation cost to simulate the OW network, the MPC-based controller does not produce any useful results in a reasonable time as explained before. We instead implemented the greedy optimization method of \textcite{krichene2018social} as a baseline. Specifically we used a genetic algorithm for the optimization with a constraint on the run time of one minute, as it is an online algorithm. It is important to note that RL policy makes a routing decision within a millisecond during test time. We compare the RL controller with this greedy method and the selfish routing.

\begin{figure}[h]
	\centering
	\includegraphics[width=\linewidth]{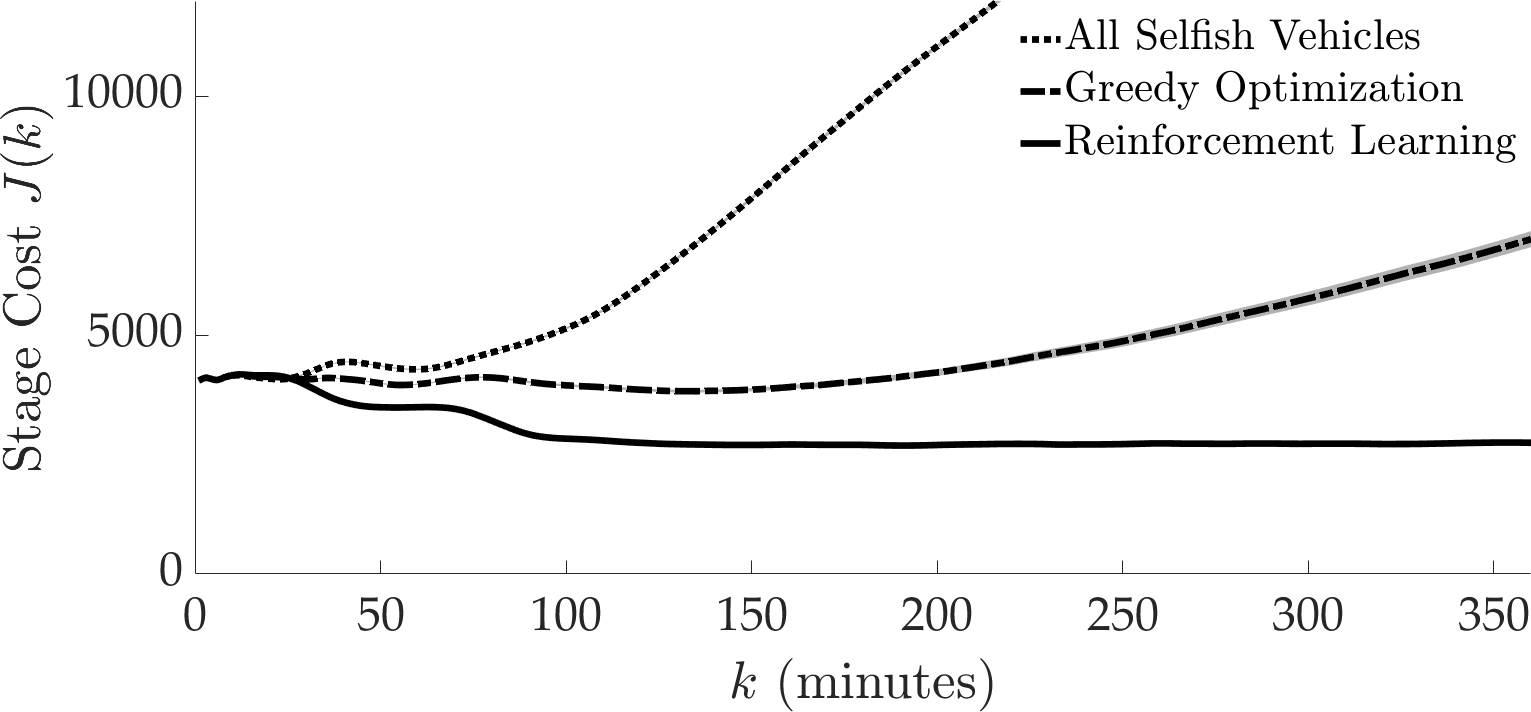}
	\caption{Time vs. number of cars under selfish, greedy and RL routing on OW network.}
	\label{fig:ow_results}
\end{figure}

Fig.~\ref{fig:ow_results} shows the number of cars in the network over time (mean $\pm$ standard error over $100$ simulations). Again, the selfish routing and the greedy optimization method of \cite{krichene2018social} suffer from linearly growing queues, while RL controller is able to stabilize the queues and keeps the network uncongested even though the network may start from a congested state. Furthermore, we check whether the reduced action space is really sufficient. We observe that, over $100$ episodes, 98.92\% of the autonomous vehicles were routed to the paths that are faster than the fastest path that is not in the action space.

To analyze the performance RL controller in comparison with the optimal equilibrium, we now move to parallel networks.

\subsection{Parallel Networks}
We consider a parallel network from downtown Los Angeles to the San Fernando Valley with $3$ paths. The highway numbers and the approximated parameter tuples (length, number of lanes, speed limit) are:
\begin{enumerate}[nosep]
	\item 110N (5 miles, 3 lanes, 60 mph); 101N (10 miles, 3 lanes for 5 miles then 2 lanes, 60 mph)
	\item 10E (5 miles, 4 lanes, 75 mph); 5N (10 miles, 4 lanes, 75 mph); 134W (5 miles, 3 lanes, 75 mph)
	\item 10W; 405N (both 10 miles, 4 lanes, 75 mph); 101S (5 miles, 3 lanes, 75 mph)
\end{enumerate}
As the cells are now not shared between the paths, we employ better latency estimates for human choice dynamics: we compute them as the actual latencies that would occur if there were no accidents and no more demand into the network.

We perform $3$ sets of experiments. In the first two, we disable accidents and analyze the effects of varying the number of paths and autonomy. As the shortest path has $15$ cells, we exclude MPC-based controller from our analysis as it is computationally prohibitive to adopt a receding horizon longer than $15$ minutes.

\begin{figure*}[t]
	\centering
	\includegraphics[width=\linewidth]{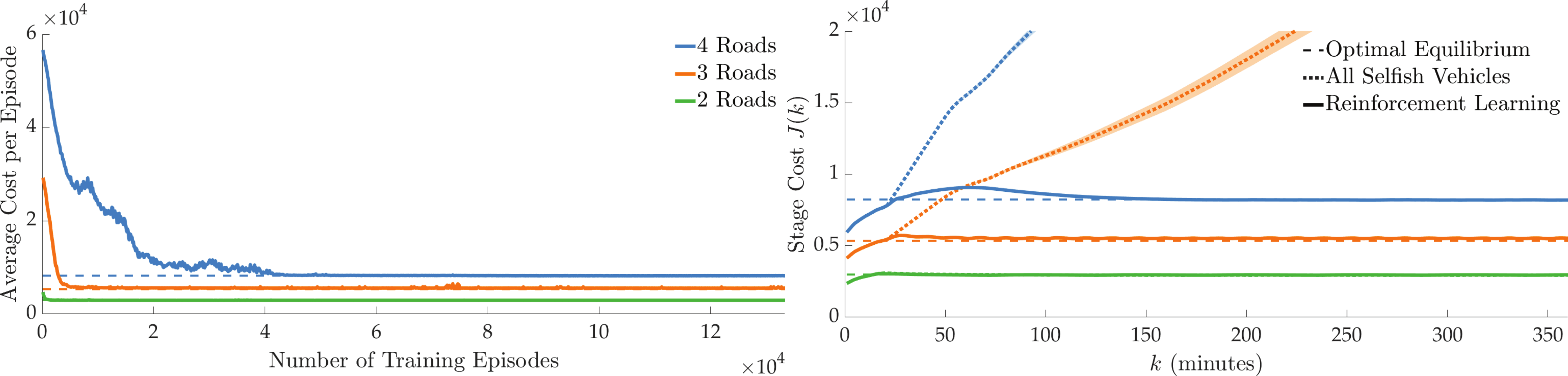}
	\caption{Varying number of paths. \textbf{(a)} Average number of cars in the system per episode during RL training. \textbf{(b)} Time vs. number of cars in the system for the comparison of selfish and RL routing in parallel networks.}
	\label{fig:experiment_fig1}
\end{figure*}

\textbf{Varying number of paths.} We first vary the number of paths $|\pathset|\!\in\!\{2,3,4\}$ by duplicating, or removing, the third path. We set the autonomy level of the demand $\bar{\autlev}\!=\!0.6$, and $\demandhum + \demandaut$ to be $95\%$ of the maximum capacity under this autonomy level.
We plot learning curves in Fig.~\ref{fig:experiment_fig1}~(a). It can be seen that even with $|\pathset| \!=\!4$ when observation space is $144$-dimensional, the agent successfully learns routing within $40$ million time steps. With randomized initial states, the agents learn routing policies that achieve nearly as good as optimal equilibrium for all $|\pathset|\!\in\!\{2,3,4\}$. In Fig.~\ref{fig:experiment_fig1}~(b), we plot the number of cars (mean $\pm$ standard error over $100$ simulations) in the system over time. While selfish routing causes congestion by creating linearly growing queues when $|\pathset|\!>\!2$, RL policies successfully stabilize queues and even reach car numbers of optimal equilibria.

\begin{figure*}[t]
	\centering
	\includegraphics[width=\linewidth]{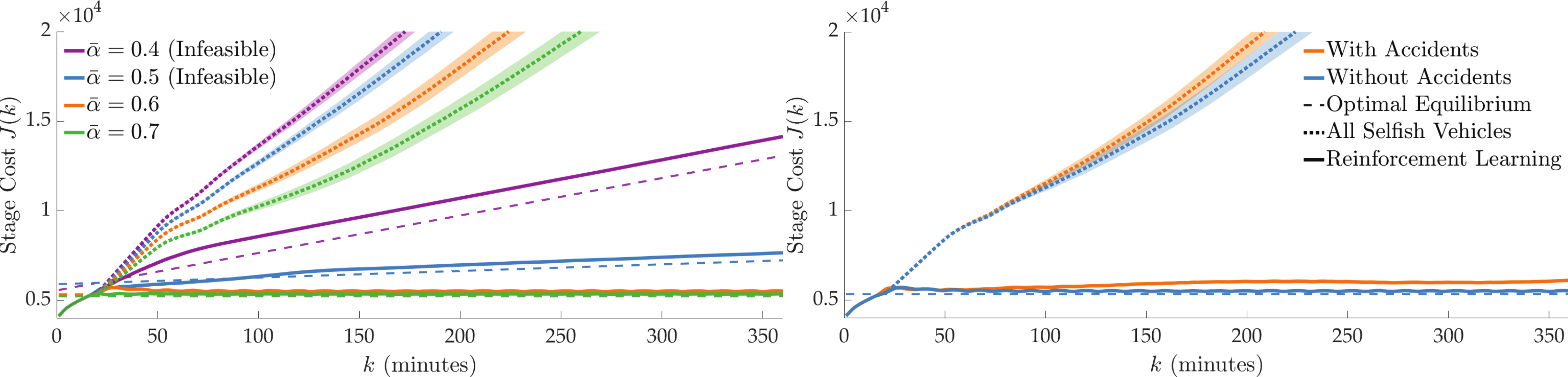}
	\caption{\textbf{(a)} Varying autonomy. \textbf{(b)} Varying the presence of accidents and noise in the demand.}
	\label{fig:experiment_fig2}
\end{figure*}

\textbf{Varying autonomy.} We take $|\pathset|\!=\!3$ and vary the autonomy of demand $\bar{\autlev}\in\{0.4,0.5,0.6,0.7\}$ without changing the total demand $\demandhum + \demandaut$. Note the demands are infeasible when $\bar{\autlev}\in\{0.4,0.5\}$.
In Fig.~\ref{fig:experiment_fig2}~(a), we plot the number of cars (mean $\pm$ standard error over $100$ simulations) in the system over time. The result is similar to the previous experiment when the demand is feasible. With infeasible demand, RL agent keeps a queue that is only marginally longer than the queue that optimal equilibrium would create. On the other hand, selfish routing grows the queue with much faster rates. These experiments show RL policy successfully handles random initializations.

\begin{figure*}[h]
	\centering
	\includegraphics[width=\linewidth]{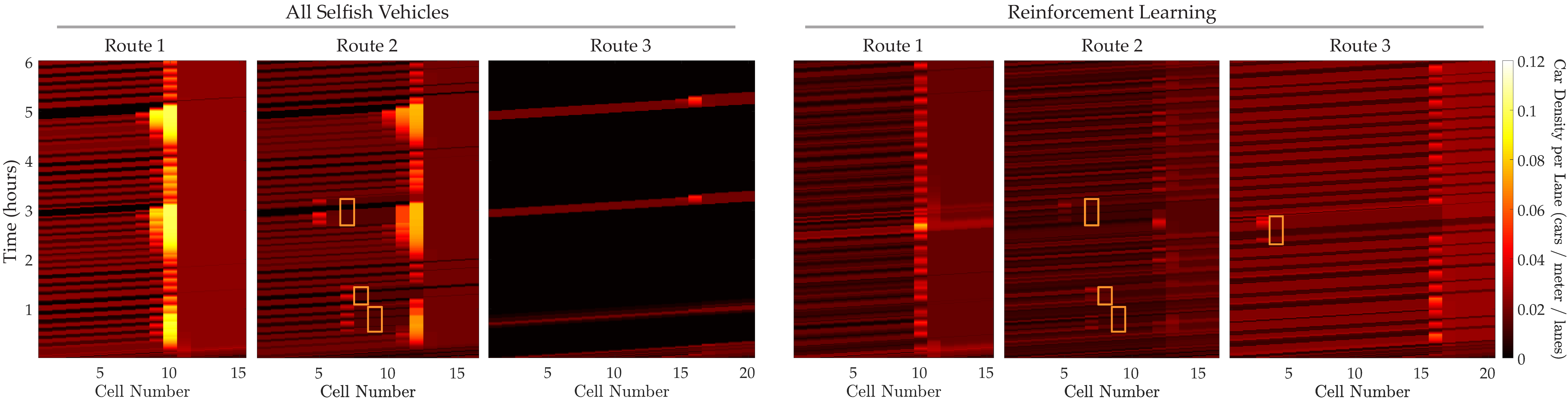}
	\caption{Space-time diagrams on a parallel traffic network with accidents and noisy demand. Orange rectangles represent accidents.}
	\label{fig:space_time_diagram}
\end{figure*}

\textbf{Accidents.} In the third set, we fix $|\pathset|=3$ and $\bar{\autlev}=0.6$ for the same total average demand and enable accidents. As before, the expected frequency of accidents is $1$ per $100$ minutes, and clearing out an accident takes $30$ minutes on average.
Fig.~\ref{fig:experiment_fig2}~(b) shows the RL policy successfully handles accidents, indicating a good generalization performance by the RL controller. To give a clearer picture, we provide the space-time diagrams and the detailed information about the system states of a sample run in Figs.~\ref{fig:space_time_diagram} and \ref{fig:experiment_fig3}, respectively. Fig.~\ref{fig:space_time_diagram} shows that selfish routing causes congestion by not utilizing the third route, whereas RL can avoid congestion and handle accidents. Fig.~\ref{fig:experiment_fig3} shows the number of cars in each cell as well as the queue lengths over time. The small oscillations, which occur even after the effect of the accidents disappear (between third and fourth hours), are due to noisy demand and the discretization of cells. With selfish routing, the vehicles use the longest path only when there is an accident in another path (around first and third hours) or the other two paths are congested (third and fifth hours). In contrast, RL makes good use of the network and leads to altruistic behavior. It also handles the accidents by effectively altering the routing of autonomous cars (around third hour, autonomous cars start using the first route until the accident in the third route is cleared). Hence, it manages to stabilize the queue and prevent congestion. We provide video visualizations of this run at \url{https://youtu.be/XwdSJuUb09o}.

\begin{figure*}[h]
	\centering
	\includegraphics[width=\linewidth]{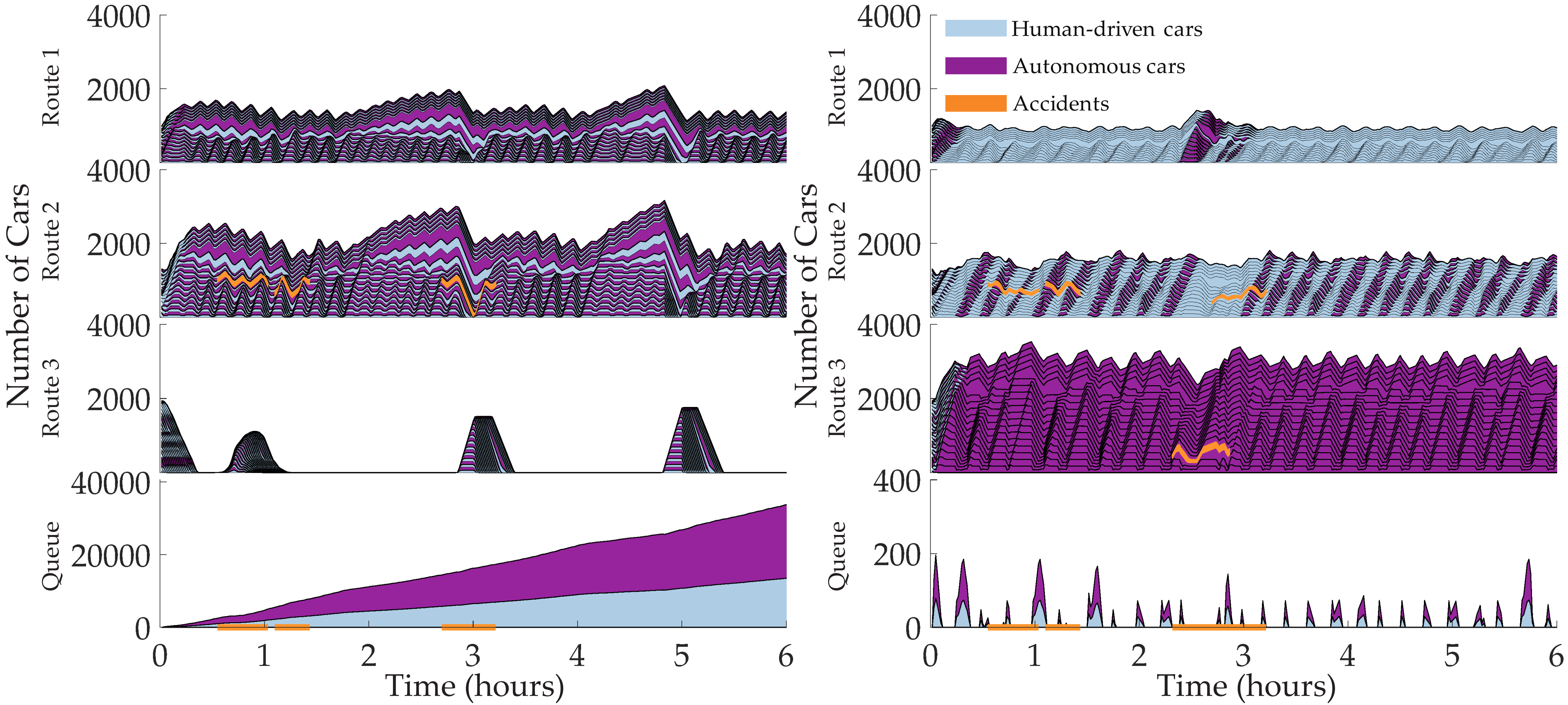}
	\caption{The network under perturbations due to accidents and noisy demand. For each path and time step, from bottom to top, the stacked color segments show the number of cars in the cells from origin to the destination. Congestion occurs only upstream to the bottlenecks. \textbf{(a)} Selfish routing. \textbf{(b)} RL routing.}
	\label{fig:experiment_fig3}
\end{figure*}

\section{Conclusion}
\label{sec:conclusion}
\textbf{Summary.} We presented a framework for understanding a dynamic traffic network shared between selfish human drivers and controllable autonomous cars. We show, using deep RL, we can find a policy to minimize the average travel time experienced by users of the network. We develop theoretical results to describe and calculate the best equilibria that can exist and empirically show that our policy reaches the best possible equilibrium performance in parallel networks. Further, we provide case studies showing how the training period scales with the number of paths, and we show our control policy is empirically robust to accidents and stochastic demand.

\textbf{Limitations.} We used the number of cars in each cell as predictive features for RL training. Although this makes the state space dimensionality grow only linearly with the number of cells, it may not be scalable to much larger traffic networks. Moreover, the action space grows linearly with the number of source-destination pairs, also impacting the scalability of the algorithm.

\textbf{Future work.} This work opens up many future directions for research, including using multi-agent reinforcement learning to model autonomous vehicles with competitive goals and/or en route decision making ability, and improving how the training time scales with the complexity of the network. Another interesting future work is to investigate how an RL policy can be deployed and the simulation imperfections (including the dependency on the simulated human choice dynamics) can be alleviated by collecting online data using sensors from the real traffic network.

\section*{Acknowledgments}
This work was supported by NSF grant \#1953032 and Toyota. Toyota Research Institute (TRI) provided funds to assist the authors with their research but this article solely reflects the opinions and conclusions of its authors and not TRI or any other Toyota entity.

\renewcommand*{\bibfont}{\small}
\printbibliography

\section{Appendix}
\label{sec:appendix}
\balance
\subsection{Summary of notation}\label{sct:notation}
See Table~\ref{tab:notation}.
\begin{table}[h]
	\setlength{\tabcolsep}{3pt}
	\caption{Summary of Notation}\label{tab:notation}
	\centering
	\begin{tabular}{ |l l l|}
		\hline
		$\pathidx $ & Path index & unitless \\ 
		$\pathset $ & Set of paths in the network & set of paths\\
		$ \cellidx $ & Cell index & unitless \\
		$\cellset$ & set of cells in the network & set of cells \\
		$\cellset_\pathidx$ & set of cells in path $\pathidx$ & set of cells \\
		$\upstreamcells_\cellidx$ & set of cells upstream of cell $\cellidx$ & set of cells \\
		$\ffvelocity_{\cellidx} $ & Free-flow velocity of cell $\cellidx$ & cells/time step \\
		$\numlanes_{\cellidx} $ & Number of lanes of cell $\cellidx$ & unitless \\
		$\spacehuman_{\cellidx} \; (\spaceaut_{\cellidx})$ & Nominal vehicle headway on cell $\!\cellidx$ & cells/vehicle \\
		$\denshum_{\cellidx} \; (\densaut_{\cellidx}) $ & Density of vehicles on cell $\cellidx$ & vehicles/cell \\
		$\dens_{\cellidx} $ & Total vehicle density on cell $\cellidx$ & vehicles/cell \\
		$\flowhum_{\cellidx} \; (\flowaut_{\cellidx}) $ & Flow of vehicles from cell $\cellidx$ & vehicles/time step \\
		$\flowinhum_\cellidx \; (\flowinaut_\cellidx)$ & Hum. (aut.) veh flow into cell $\cellidx$ & vehicles/time step   \\
		$\autlev_{\cellidx} $ & Autonomy level of cell $\cellidx$  & unitless \\
		$\critdens_{\cellidx}(\autlev) $ & Critical density of cell $\cellidx$, at aut. $\autlev$ & vehicles/cell \\
		$\jamden_{\cellidx}$ & Jam (maximum) density of cell $\cellidx$ & vehicles/cell \\
		$ \capacity_{\cellidx}(\autlev) $ & Capacity of cell $\cellidx$, at aut. $\autlev$ &  vehicles/time step \\
		$\shockspd_{\cellidx}(\autlev) $ & Shockwave speed of cell $\cellidx$, at aut. $\autlev$  & cells/time step \\
		$ \timeidx$ & Time index & unitless \\
		$ \latency_\pathidx(\timeidx) $ & Latency of path $\pathidx$ if starting at time $\timeidx$ & time steps \\
		$\priority_{\cellidx}(\timeidx)$ & Priority for cell $\cellidx$ at a merge at time $\timeidx$ & unitless \\
		$ \routinghum_\cellidx(\pathidx, \timeidx) \; (\routingaut)$ & Frac. of hum. (aut.) vehs in $\cellidx$ on $\pathidx$ at $\timeidx$ & unitless \\
		$ \fracouthum_\cellidx(\cellidx',\timeidx) \; (\fracoutaut) $ & Frac. of hum. (aut.) vehs $\cellidx \rightarrow \cellidx'$ at $\timeidx$ & unitless \\
		$J(\timeidx)$ & Stage cost at time $\timeidx$ & vehicles \\
		$ \numbcells \; (\numnbcells) $ & \# of (non)bottleneck cells on path $\pathidx$ & cells \\ 
		$\numlanesb \; (\numlanesnb) $ & \# of lanes in (non)bottleneck cells on $\pathidx$ & unitless \\
		$\laneratio_\pathidx $ & $:= \numlanesb/\numlanesnb $ & unitless \\
		$\gamma_\pathidx $ & Number of congested cells on path $\pathidx$ & cells \\
		\hline
	\end{tabular}
\end{table}

\subsection{Proofs for Section \ref{sct:network_eq}}\label{sct:network_proofs}

\textbf{Proof of Lemma \ref{lma:homogeneousautlevel}.}
	By definition, at equilibrium, the number of vehicles in each cell $\cellidx$ in $\cellset_\pathidx$, $\densaut_{\cellidx}(\timeidx)$ and $\densaut_{\cellidx}(\timeidx)$ is constant for all times $\timeidx$. Since by definition the incoming flow is also constant, by the definition of the sending and receiving functions, constant cell densities implies constant flows. By \eqref{eq:dens_update}, a constant density also implies that the incoming and outgoing flow in each cell are equal. This means that all cells will have the same incoming flow as the first cell. Further, we know that since the density of autonomous vehicles is constant over time, incoming and outgoing autonomy levels are equal. Accordingly, if cell $\cellidx'$ is the cell immediately upstream of $\cellidx$, then $\autlev_{\cellidx'}(\timeidx) \flow_{\cellidx'}(\timeidx) = \autlev_{\cellidx}(\timeidx) \flow_{\cellidx}(\timeidx)$. Since we also have $\flow_{\cellidx'}(\timeidx)=\flow_{\cellidx}(\timeidx)$, this implies that $\autlev_{\cellidx'}(\timeidx)=\autlev_{\cellidx}(\timeidx)$. Therefore the autonomy level of all cells is the same. Let us denote this uniform autonomy level $\autlev_{\pathidx}$. Let the index of the first cell in the path be 0. Then, $\demandhum_{\pathidx} + \demandaut_{\pathidx} = \flow_{0}$ and $\demandaut_{\pathidx} = \autlev_{\pathidx} \flow_{0}$. Combining these two expressions, we find $\autlev_{\pathidx} = \demandaut_{\pathidx}/(\demandhum_{\pathidx} + \demandaut_{\pathidx})$. \qed
	
\textbf{Proof of Lemma \ref{lma:ff}.}
	Under Assumption~\ref{assump:ff_latency}, no two paths have the same free-flow latency. With Proposition~\ref{prop:equilibrium}, this implies that if an equilibrium has a used path with no congestion, it must be the used path with greatest free-flow latency, as otherwise all used paths would not have the same latency. Therefore, if an equilibrium routing with positive flow on paths ${[}\pathidx{]}$ has a path in free-flow, it must be path $\pathidx$. Otherwise, we can construct an equilibrium with the same demand that has path $\pathidx$ in free-flow. Recall that the latency on paths in equilibrium is increasing with the length of the congested portion of the path, $\gamma_{\pathidx'}$, and $\gamma_{\pathidx'}=0$ corresponds to an uncongested path. If all paths are congested, we consider decreasing the length of congestion on all paths simulatenously, at rates which keep the path latencies equal. This continues until path $\pathidx$ becomes completely uncongested. This construction proves the lemma. \qed

\subsection{Overview of Proximal Policy Optimization (PPO)}\label{sct:ppo_overview}
In this section, we give a brief overview of the PPO method \cite{schulman2017proximal} we used for training our deep reinforcement learning model. We first start with formalizing the problem. We then introduce the policy gradients and the details of PPO. To keep the notation consistent with the reinforcement learning literature, we abuse the notation for some variables. Hence, this section of the appendix is written in a standalone way, and the variables should not be confused with the notation introduced in the main paper (e.g. $f$ is going to denote the transition distribution of the system as introduced below, instead of flow values as in the main paper).

\textbf{Problem Setting.} We consider a sequential decision making problem in a Markov decision process (MDP) represented by a tuple $(\mathcal{S}, \mathcal{A}, f, T, r, \gamma)$, where $\mathcal{S}$ is the set of states. $\mathcal{A}$ denotes the set of actions, and the system transitions with respect to the transition distribution $f: \mathcal{S}\times\mathcal{A}\times\mathcal{S}\to[0,1]$. For example, if $f(s, a, s')=p$, this means taking action $a\in\mathcal{A}$ at state $s\in\mathcal{S}$ transitions the system into state $s'$ with probability $p$. Next, $T$ denotes the horizon of the system, i.e., the process gets completed after $T$ time steps. The reward function $r:\mathcal{S}\times\mathcal{A}\to\mathbb{R}$ maps state-actions to reward values. The decision maker is then trying to maximize the cumulative reward over $T$ time steps by only observing the observations (not states). Finally $\gamma$ is a discount factor that sets how much priority we give to optimizing earlier rewards in the system.

Let us now describe how we formulate a transportation network with the CTM model as an MDP in this paper. The state of the network is fully defined by the following information:
\begin{itemize}
	\item Location of each vehicle (which cell or queue it is in),
	\item Type of each vehicle (human-driven or autonomous),
	\item Accident information (where and when it happened), and
	\item Planned path of each vehicle (which cells it is going to traverse).
\end{itemize}
In our model, we assumed the first three items in the above list are available as observations. While this breaks the Markov assumption, deep RL techniques often perform well in partially observable MDPs, too. So our deep RL policy is trying to make its decisions based only on those first three observations, and the non-observability of the planned paths increases the stochasticity of the problem. The action set of the decision maker is defined by the set of available routing paths of autonomous vehicles. The transition distribution follows the dynamics of CTM, human choice dynamics, as well as the accidents which also introduce stochasticity into the system. Finally, as a reward function, one can think of using the negative of the number of cars in the system as a proxy to negative of overall latency in the network.

\textbf{Policy Gradients.} To solve this problem using deep neural networks, we model the decision-maker agent with a stochastic policy $\pi_\theta$ parameterized with $\theta$ (e.g. weights of the neural network), such that $\pi_\theta(a \mid s)$ gives the probability of taking action $a$ when observing state $s$. The goal of the agent is to maximize the expected cumulative discounted reward:
\begin{align*}
J(\theta) = \mathbb{E}_{\tau\sim{\pi_\theta}}\left[\sum_{t=0}^{T-1}\gamma^t r(s_t,a_t)\right]
\end{align*}
where $\tau$ denotes a trajectory $(s_0,a_0,\dots,s_{T-1},a_{T-1},s_T)$ in the system. The discount factor is to improve robustness and to reduce susceptibility against high variance. We can equivalently write this objective as:
\begin{align*}
J(\theta) = \int_{\Xi} \pi_\theta(\tau) r(\tau)d\tau
\end{align*}
where $\Xi$ is the set of all possible trajectories, $\pi_\theta(\tau)$ is the probability of trajectory $\tau$ under policy $\pi_\theta$, and $r(\tau)$ is the cumulative discounted reward of trajectory $\tau$. The idea in policy gradients is to take gradient steps to maximize this quantity by optimizing $\theta$:
\begin{align*}
\nabla_{\theta} J(\theta) &= \nabla_{\theta} \int_{\Xi} \pi_\theta(\tau) r(\tau)d\tau\\
&= \int_{\Xi} \nabla_{\theta} \pi_\theta(\tau) \frac{\pi_\theta(\tau)}{\pi_\theta(\tau)} r(\tau)d\tau\\
&= \int_{\Xi} \pi_\theta(\tau) r(\tau) \nabla_{\theta} \log\pi_\theta(\tau)d\tau\\
&= \mathbb{E}_{\tau\sim\pi_\theta}\left[r(\tau) \nabla_{\theta} \log\pi_\theta(\tau)\right]
\end{align*}
which we can efficiently approximate by sampling trajectories using the policy.

Unfortunately, this vanilla policy gradient method is not robust against variance (due to stochasticity in the environment and trajectory sampling) and suffers from data-inefficiency. In recent years, several works have developed alternative ways to approximate the gradients. One such idea is based on using baselines to reduce variance:
\begin{align*}
\nabla_{\theta} J(\theta) = \mathbb{E}_{\tau\sim\pi_\theta}\left[\sum_{t=0}^{T-1} \nabla_{\theta} \log\pi_{\theta}(a_t^\tau \mid s_t^\tau)\hat{A}_t^\tau\right]
\end{align*}
where $\hat{A}$ is called the estimated advantage function, which is usually defined as $G_t^\tau-V(s_t^\tau)$, where $G_t^\tau$ is the cumulative discounted reward of the trajectory $\tau$ after (and including) time step $t$, and $V(s_t^\tau)$ is some baseline that quantifies the value of state $s_t^\tau$. This new equation for $\nabla_{\theta} J(\theta)$ holds due to the Markov assumption and that the baseline is independent from the policy parameter $\theta$.

Having presented the policy gradients and the use of baselines for variance reduction, we are now ready to give an overview of PPO.

\textbf{Proximal Policy Optimization (PPO).} PPO further improves the robustness and data-efficiency of policy gradient methods by using a surrogate objective that prevents the policy from being updated with large deviations. Instead of the usual objective $\mathbb{E}_{\tau\sim\pi_\theta}\left[\log\pi_\theta(a_t^\tau \mid s_t^\tau) \hat{A}_t^\tau \right]$, PPO uses the following objective:
\begin{align*}
J_1(\theta)=\mathbb{E}_{\tau\sim\pi_\theta}\left[\min(g^\tau_t(\theta)\hat{A}_t^\tau,\textrm{clip}(g^\tau_t(\theta),1-\epsilon,1+\epsilon)\hat{A}_t^\tau)\right]
\end{align*}
where 
\begin{align*}
g^\tau_t(\theta) = \!\frac{\pi_\theta(a^\tau_t \mid s^\tau_t)}{\pi_{\theta_\textrm{old}}(a^\tau_t \mid s^\tau_t)}\:\textrm{and}\:\textrm{clip}(x,\epsilon_1,\epsilon_2)\!=\!\begin{cases}
\epsilon_1 & x<\epsilon_1,\\
x & \epsilon_1\leq x \leq \epsilon_2,\\
\epsilon_2 & \textrm{otherwise}.
\end{cases}
\end{align*}

In addition to $J_1(\theta)$, PPO uses two more objective functions and converts the problem into a multi-objective optimization problem. The first additional objective is for the baseline $V(s_t^\tau)$. Specifically, PPO learns a parameterized value function $V_\phi$ in a supervised way to minimize $(V_{\phi}(s_t^\tau)-V_t^\textrm{target})^2$ where $V_t^\textrm{target}$ is calculated using the sampled trajectories as a sum of discounted rewards after (and including) time step $t$. It should be noted that this does not make $G_t^\tau-V_\phi(s_t^\tau)=0$, because $V_\phi(s_t^\tau)$ is an estimate of the true value function and is updated after the computation of the estimated advantage. Therefore,
\begin{align*}
J_2(\phi)=-\mathbb{E}_{\tau\sim\pi_\theta}\left[V_{\phi}(s_t^\tau)-V_t^\textrm{target}\right]\:.
\end{align*}

Finally, PPO uses an entropy bonus (inspired by \cite{mnih2016asynchronous}) to ensure sufficient exploration:
\begin{align*}
J_3(\theta)=\mathbb{E}_{\tau\sim\pi_\theta}H(\pi_\theta(\cdot \mid s_t^\tau))\:,
\end{align*}
where $H$ is information entropy. At the end, PPO tries to solve:
\begin{align*}
\textrm{maximize}_{\theta,\phi} \quad J_1(\theta) + J_2(\phi) + c J_3(\theta)
\end{align*}
where $c$ is the coefficient for the entropy term.

\subsection{Experiment details}
\label{sct:ppo_params}
In implementation, we used $J(\timeidx)-J(\timeidx-1)$ as a proxy cost for time step $\timeidx$, where $J(0)=0$.

Below are the set of hyperparameters we used for PPO. We refer to Section~\ref{sct:ppo_overview} and \cite{schulman2017proximal} for the definitions of PPO-specific parameters. While this set yields good results as we presented in the paper, a careful tuning may improve the performance.
\begin{itemize}[nosep]
	\item Number of Time Steps: $40$ million
	\item Number of Actors: $32$ ($32$ CPUs in parallel)
	\item Time Steps per Episode During Training: $300$
	\item Time Steps per Actor Batch: $1200$
	\item $\epsilon$ for Clipping in the Surrogate Objective: $0.2$
 	\item Optimization Step Size (OSS): $3\times 10^{-4}$
	\item Annealing for $\epsilon$ (Clipping) and OSS: Linear (down to $0$)
	\item Entropy Coefficient: $0.005$
	\item Number of Optimization Epochs: $5$
	\item Optimization Batch Size: $64$
	\item $\gamma$ for Advantage Estimation: $0.99$
	\item $\lambda$ for Advantage Estimation: $0.95$
	\item $\epsilon$ for Adam Optimization: $10^{-5}$
\end{itemize}

Finally, we report the training times (for $40$ million time steps) and the number of time steps of empirical convergence (in terms of reward value) for each RL policy in Table~\ref{tab:training_times}. In test time, RL policies produce an action in under $1$ ms.

\begin{table}[h]
	\setlength{\tabcolsep}{3pt}
	\caption{Training and Convergence Times}\label{tab:training_times}
	\centering
	\begin{tabular}{ | l | c | c |}
		\hline
		\textbf{Policy} & \textbf{Training Time} & \textbf{Time Step of Convergence} \\
		\hline
 		Simple General Network & $10.0$ hours & $26.3$ million \\ 
 		OW Network & $253.1$ hours & $31.0$ million \\ 
		$\lvert\pathset\rvert=2$ & $22.2$ hours & $0.7$ million \\
		$\lvert\pathset\rvert=3$ & $38.9$ hours & $10.0$ million \\
		$\lvert\pathset\rvert=3$, w/ accidents & $40.5$ hours & $22.0$ million \\
		$\lvert\pathset\rvert=3$, $\bar\autlev = 0.4$ & $50.6$ hours & $25.5$ million \\
		$\lvert\pathset\rvert=3$, $\bar\autlev = 0.5$ & $43.1$ hours & $19.3$ million \\
		$\lvert\pathset\rvert=3$, $\bar\autlev = 0.7$ & $38.6$ hours & $6.6$ million \\
		$\lvert\pathset\rvert=4$ & $101.4$ hours & $23.3$ million \\
		\hline
	\end{tabular}
\end{table}




\end{document}